\documentclass[11pt]{article}

\usepackage[a4paper,margin=1in]{geometry}
\usepackage{amsmath,amssymb,amsfonts,amsthm}
\usepackage{mathtools}
\usepackage{mathrsfs}
\usepackage{enumitem}
\usepackage{hyperref}
\usepackage{tikz-cd}
\usepackage[all]{xy}

\hypersetup{
    colorlinks=true,
    linkcolor=blue,
    citecolor=blue,
    urlcolor=blue
}

\newtheorem{theorem}{Theorem}[section]
\newtheorem{proposition}[theorem]{Proposition}
\newtheorem{lemma}[theorem]{Lemma}
\newtheorem{corollary}[theorem]{Corollary}

\theoremstyle{definition}

\newtheorem{example}[theorem]{Example}

\theoremstyle{remark}

\title{Koszul Duality for Coherent Sheaves}

\author{
A. M. Bouhada
}

\date{}

\begin{document}

\maketitle
\begin{center}
\emph{Dedicated to J. P. Serre on the occasion of his 100th birthday.}
\end{center}
\begin{abstract}
We establish a bounded derived Koszul duality for infinite-dimensional
Koszul algebras and derive the corresponding singular Koszul duality.
We then specialize this framework to two classes of Koszul algebras,
namely quadratic monomial algebras and absolutely Koszul algebras
satisfying an additional homological condition, for which the resulting
dualities admit particularly well-behaved forms.

As an application to algebraic geometry, let
\(X\subseteq \mathbb{P}_k^n\) be an arbitrary closed projective subscheme.
We obtain a Koszul-dual description of the bounded derived category
\(\mathsf{D}^{b}\!\bigl(\operatorname{coh}(X)\bigr)\), thereby yielding
a BGG-type correspondence for arbitrary closed projective subschemes in
projective space.

As a second application, in noncommutative projective geometry, we
consider generalized Artin--Schelter regular Koszul algebras
\(\Lambda^{!}\) arising as Koszul duals of finite-dimensional
self-injective Koszul algebras \(\Lambda\). We show that
\(\mathsf{D}^{b}\!\bigl(\operatorname{qgr}(\Lambda^{!})\bigr)\) is
triangulated equivalent to the bounded derived category of
finite-dimensional modules over a finite-dimensional Koszul algebra of
finite global dimension. This yields a Beilinson-type description of
\(\mathsf{D}^{b}\!\bigl(\operatorname{qgr}(\Lambda^{!})\bigr)\), extending
the classical description of coherent sheaves on projective space to
this noncommutative setting.
\end{abstract}
% --------------------------------------------------
% MSC AND KEYWORDS
% --------------------------------------------------
\noindent \textbf{2020 Mathematics Subject Classification.}
16E30, 18E30, 16G20, 16S37, 14F08.

\noindent \textbf{Keywords.}
Koszul duality, bounded derived categories, singularity categories, BGG correspondence,
tails category, coherent sheaves, noncommutative projective schemes, Koszul algebras, quadratic
monomial algebra, absolutely Koszul algebra.
\tableofcontents
\section{Introduction}
In 1978, Bernstein, Gel'fand, and Gel'fand discovered a remarkable
homological correspondence between two algebras of quite different
nature: the exterior algebra and the symmetric algebra~\cite{10}. The
former is finite-dimensional and self-injective, whereas the latter is a
polynomial algebra, hence infinite-dimensional and of finite global
dimension. Their correspondence provides a homological description of
coherent sheaves on projective space. More precisely, if
\(E=\bigwedge V\) is the exterior algebra on a finite-dimensional vector
space \(V\), then the bounded derived category
\(\mathsf{D}^{b}\!\bigl(\operatorname{Coh}(\mathbb{P}^{n})\bigr)\)
is triangulated equivalent to the stable category of finite-dimensional
graded \(E\)-modules. This is the classical
Bernstein--Gel'fand--Gel'fand correspondence, or BGG correspondence. It
has become a fundamental tool in algebraic geometry and commutative
algebra, with applications to syzygies, free resolutions, Chow forms,
Betti tables, and the cohomology of vector bundles; see, for
example,~\cite{24,27,28}. More recently, a BGG-type correspondence was
obtained for finite-dimensional Iwanaga--Gorenstein Koszul algebras; see
\cite[Theorem~3.17]{15}.

A different description of
\(\mathsf{D}^{b}\!\bigl(\operatorname{Coh}(\mathbb{P}^{n})\bigr)\)
is due to Beilinson~\cite{6}. His theorem realizes this category as the
bounded derived category of finite-dimensional modules over a
finite-dimensional Koszul algebra of finite global dimension. Thus the
BGG and Beilinson descriptions provide two complementary homological
realizations of the derived category of coherent sheaves on projective
space.

The BGG correspondence was subsequently extended beyond the classical
pair consisting of the symmetric and exterior algebras. Important steps
in this direction were taken by Beilinson, Ginzburg, and
Schechtman~\cite{7}, and a general formulation was later developed by
Beilinson, Ginzburg, and Soergel~\cite{8}. This led to the modern theory
of Koszul duality. Given a Koszul algebra \(\Lambda\), with Koszul dual
\(\Lambda^{!}\), the Beilinson--Ginzburg--Soergel theory relates suitable
graded derived categories of \(\Lambda\) and \(\Lambda^{!}\). In the form
relevant here, one has a triangulated equivalence
\[
\mathsf{D}^{\downarrow}\!\bigl(\Lambda^{!}\textup{-}\mathrm{GMod}\bigr)
\simeq
\mathsf{D}^{\uparrow}\!\bigl(\Lambda\textup{-}\mathrm{GMod}\bigr).
\]
Koszul duality has become a fundamental tool in representation theory, homological algebra, algebraic geometry, and noncommutative geometry. A central feature of the theory is its ability to relate algebras with substantially different algebraic and homological properties. For example, the Koszul dual of a commutative Koszul algebra is typically noncommutative, while the dual of a finite-dimensional Koszul algebra is often infinite-dimensional. Thus, Koszul duality provides a natural bridge between commutative and noncommutative geometry, as well as an effective framework for transferring homological information between finite-dimensional and infinite-dimensional settings.

In the finite-dimensional setting, bounded forms of Koszul duality are
well understood; see, for example,~\cite[Theorem~3.8]{15}. For
infinite-dimensional Koszul algebras, however, the bounded theory is
considerably more delicate. In particular, without assumptions such as
finite-dimensionality, noetherianity, or coherence on \(\Lambda\) or
\(\Lambda^{!}\), the classical Koszul duality does not directly provide
a bounded derived equivalence between the natural categories of
finitely generated or finitely presented modules.

The main purpose of this paper is to develop a bounded derived Koszul
duality that applies to infinite-dimensional Koszul algebras without
imposing such finiteness assumptions on the algebra or its Koszul dual.
The basic objects on the two sides are the categories of almost linear
and almost colinear graded modules. We prove that the classical Koszul
duality functors restrict to mutually inverse triangulated equivalences
between their bounded derived categories. This also yields the
corresponding singular Koszul duality. We then specialize the general
construction to two classes for which the resulting equivalences admit
particularly concrete descriptions: quadratic monomial algebras and
absolutely Koszul algebras satisfying an additional homological
condition.

A principal motivation for this bounded theory comes from algebraic geometry. Many naturally occurring infinite-dimensional Koszul algebras are homogeneous coordinate rings of projective varieties. Examples include graded quotients of polynomial algebras arising from Grassmannians, Schubert varieties, and more general projective homogeneous spaces. Such algebras occur naturally in the study of syzygies, Castelnuovo--Mumford regularity, and derived categories of coherent sheaves. Let \(X=\operatorname{Proj}(\Lambda)\), where \(\Lambda\) is a commutative
noetherian graded algebra generated in degree \(1\). By Serre's
theorem~\cite{52}, there is an equivalence
\(\operatorname{coh}(X)\simeq\operatorname{qgr}(\Lambda)\), where
\(\operatorname{qgr}(\Lambda)
=\Lambda\textup{-}\mathrm{Fp}^{\mathbb Z}/
\Lambda\textup{-}\mathrm{gmod}\). Consequently,
\(\mathsf{D}^{b}\!\bigl(\operatorname{coh}(X)\bigr)
\simeq
\mathsf{D}^{b}\!\bigl(\operatorname{qgr}(\Lambda)\bigr)\).

More generally, if \(X\) is any closed projective subscheme of a
projective space, then its homogeneous coordinate ring is a quotient
of a polynomial algebra. By a theorem of Backelin~\cite{5}, for all
sufficiently large \(d\), the Veronese algebra \(\Lambda^{(d)}\) is
Koszul. Since
\(\operatorname{Proj}(\Lambda^{(d)})
\cong\operatorname{Proj}(\Lambda)\) \cite{52},
Serre's theorem yields
\(\operatorname{qgr}(\Lambda^{(d)})
\simeq\operatorname{coh}(X)\), and hence
\(\mathsf{D}^{b}\!\bigl(\operatorname{qgr}(\Lambda^{(d)})\bigr)
\simeq
\mathsf{D}^{b}\!\bigl(\operatorname{coh}(X)\bigr)\).

Applying our bounded Koszul duality to \(\Lambda^{(d)}\), we therefore
obtain a Koszul-dual description of the bounded derived category
\(\mathsf{D}^{b}\!\bigl(\operatorname{coh}(X)\bigr)\) for arbitrary
closed projective subschemes of projective space. Since the Koszul dual
is typically a quotient of a free associative algebra, the resulting
equivalence relates a geometric category of coherent sheaves to a
triangulated category naturally associated with a noncommutative
algebra.

There is a parallel application in noncommutative projective geometry.
Artin and Zhang~\cite{2} introduced \(\operatorname{qgr}(\Lambda)\) as
the noncommutative analogue of the category of coherent sheaves on a
projective scheme. We consider generalized Artin--Schelter regular
Koszul algebras \(\Lambda^{!}\) arising as Koszul duals of
finite-dimensional self-injective Koszul algebras \(\Lambda\). In this
setting, we obtain a description of
\(\mathsf{D}^{b}\!\bigl(\operatorname{qgr}(\Lambda^{!})\bigr)\)
as the bounded derived category of finite-dimensional modules over a
finite-dimensional Koszul algebra \(B\) of finite global dimension. In
particular, this gives a Beilinson-type description of
\(\mathsf{D}^{b}\!\bigl(\operatorname{qgr}(\Lambda^{!})\bigr)\), extending
the classical picture for projective space to a noncommutative setting.

\medskip

\noindent\textbf{Organization of the paper.}
Section~2 fixes the notation and recalls the necessary background on
graded modules, Koszul algebras, triangulated categories, and the
classical Koszul duality functors. In Section~3, we establish the bounded
derived Koszul duality for infinite-dimensional Koszul algebras and
derive the corresponding singular duality. We then specialize these
results to quadratic monomial algebras and to absolutely Koszul algebras
satisfying an additional homological condition. Section~4 is devoted to
geometric applications. We first obtain a BGG-type description of the
bounded derived category of coherent sheaves on all closed projective
subschemes of projective space. We then treat
generalized AS-regular Koszul algebras arising as Koszul duals of
finite-dimensional self-injective Koszul algebras and obtain the
corresponding Beilinson-type description in noncommutative projective
geometry.
\section{Preliminaries and Notation}

Throughout the paper, \(k\) denotes a field and \(\Lambda\) a graded
\(k\)-algebra of the form
\[
\Lambda=kQ/I,
\]
where \(Q\) is a finite quiver and \(I\) is a homogeneous ideal of
\(kQ\). Unless explicitly stated otherwise, \(I\) is not assumed to be
admissible. The grading on \(\Lambda\) is induced by the path-length
grading on \(kQ\), so that
\[
\Lambda=\bigoplus_{i\geq 0}\Lambda_i.
\]

Following the standard convention in Koszul duality, we regard
\(\Lambda\) as a graded \(k\)-linear category whose objects are the
vertices of \(Q\) and whose morphism spaces are
\[
\Lambda(x,y)=e_y\Lambda e_x
=\bigoplus_{i\geq 0} e_y\Lambda_i e_x,
\]
with composition induced by multiplication in \(\Lambda\). For
simplicity, we assume that \(\Lambda\) is locally finite-dimensional,
that is,
\[
\dim_k e_y\Lambda_i e_x<\infty
\qquad
\text{for all } x,y\in Q_0 \text{ and } i\geq 0.
\]

Unless otherwise specified, all modules are left modules and are
regarded as covariant graded \(k\)-linear functors. When \(\Lambda\) is
Koszul, its Koszul dual is denoted by \(\Lambda^{!}\).

The definitions and results below extend, with the evident minor
modifications, to locally finite quivers and, more generally, to the
graded algebras considered in \cite{8}. We restrict to finite quivers
in order to streamline the exposition and to focus on the classes of
algebras most relevant to the applications developed in this paper,
notably quotients of polynomial algebras and quotients of free
associative algebras arising in geometric contexts.

\subsection{Graded Modules}

We begin by fixing the notation and finiteness conditions for graded modules that will be used throughout the paper. We regard \(\Lambda\) as
a graded \(k\)-category whose objects are the vertices of the quiver \(Q\), so
that graded modules may be viewed equivalently as graded \(k\)-linear functors.
Thus, a graded left \(\Lambda\)-module is a covariant graded \(k\)-linear
functor
\[
M\colon \Lambda \longrightarrow \mathsf{GrVect}_{k},
\]
where \(\mathsf{GrVect}_{k}\) denotes the category of \(\mathbb{Z}\)-graded
\(k\)-vector spaces.

For a vertex \(x\in Q_0\), we write \(M(x)=e_xM\). Then \(M(x)\) is a graded
\(k\)-vector space, and we write
\[
M(x)=\bigoplus_{i\in\mathbb{Z}} M(x)_i,
\qquad
M(x)_i=e_xM_i .
\]

We denote by \(\Lambda\textup{-GMod}\) the category of locally
finite-dimensional graded left \(\Lambda\)-modules, with morphisms the
degree-preserving \(\Lambda\)-linear maps. Thus, for every object
\(M\in\Lambda\textup{-GMod}\),
\[
\dim_k M(x)_i<\infty
\qquad
\text{for all }x\in Q_0\text{ and all }i\in\mathbb{Z}.
\]

The category \(\Lambda\textup{-GMod}\) is equipped with the grading shift
functor
\[
\langle 1\rangle\colon
\Lambda\textup{-GMod}\longrightarrow\Lambda\textup{-GMod}.
\]
For \(i\in\mathbb{Z}\), the shifted module \(M\langle i\rangle\) is defined by
\[
(M\langle i\rangle)_n=M_{n+i}
\qquad
(n\in\mathbb{Z}).
\]
Equivalently,
\[
(M\langle i\rangle)(x)_n=M(x)_{n+i}
\]
for every \(x\in Q_0\).

A graded module \(M\) is said to be right bounded if there exists an integer
\(N\) such that \(M(x)_i=0\) for all \(x\in Q_0\) and all \(i>N\). Dually,
\(M\) is said to be left bounded if there exists an integer \(N\) such that
\(M(x)_i=0\) for all \(x\in Q_0\) and all \(i<N\). We denote by \(\Lambda\textup{-GMod}^{-}\) and
\(\Lambda\textup{-GMod}^{+}\) the full subcategories of
\(\Lambda\textup{-GMod}\) consisting respectively of right bounded and left
bounded graded modules.

A graded module is said to be finite-dimensional if it is both left bounded and
right bounded. We denote by \(\Lambda\textup{-gmod}\) the full subcategory of
\(\Lambda\textup{-GMod}\) consisting of finite-dimensional graded modules.

For \(x\in Q_0\), let \(P_x=\Lambda e_x\) be the indecomposable projective left
\(\Lambda\)-module associated with \(x\). We denote by
\(\Lambda\textup{-Proj}^{\mathbb Z}\) the additive full subcategory of
\(\Lambda\textup{-GMod}^{+}\) generated by the objects
\(P_x\langle n\rangle\), with \(x\in Q_0\) and \(n\in\mathbb Z\).

Dually, let \(I_x=D(e_x\Lambda)\), where \(D=\operatorname{Hom}_k(-,k)\), be
the indecomposable injective left \(\Lambda\)-module associated with \(x\). We
denote by \(\Lambda\textup{-Inj}^{\mathbb Z}\) the additive full subcategory of
\(\Lambda\textup{-GMod}^{-}\) generated by the objects
\(I_x\langle n\rangle\), with \(x\in Q_0\) and \(n\in\mathbb Z\).

If \(\Lambda\) is finite-dimensional, we write
\(\Lambda\textup{-proj}^{\mathbb Z}\) and
\(\Lambda\textup{-inj}^{\mathbb Z}\) for the full subcategories of
finite-dimensional graded projective modules and finite-dimensional graded
injective modules, respectively.

A graded \(\Lambda\)-module \(M\) is said to be finitely generated if there
exists an epimorphism
\[
P^0\longrightarrow M\longrightarrow 0,
\]
where \(P^0\) is a finite direct sum of graded shifts of indecomposable
projective modules. Equivalently,
\[
P^0\cong
\bigoplus_{\ell=1}^r P_{x_\ell}\langle n_\ell\rangle
\]
for some \(x_\ell\in Q_0\) and \(n_\ell\in\mathbb Z\). We denote by
\(\Lambda\textup{-Fg}^{\mathbb Z}\) the full subcategory of
\(\Lambda\textup{-GMod}^{+}\) consisting of finitely generated graded modules.

Dually, a graded \(\Lambda\)-module \(M\) is said to be finitely cogenerated if
there exists a monomorphism
\[
0\longrightarrow M\longrightarrow I^0,
\]
where \(I^0\) is a finite direct sum of graded shifts of indecomposable
injective modules. Equivalently,
\[
I^0\cong
\bigoplus_{\ell=1}^s I_{x_\ell}\langle n_\ell\rangle
\]
for some \(x_\ell\in Q_0\) and \(n_\ell\in\mathbb Z\). We denote by
\(\Lambda\textup{-Fcg}^{\mathbb Z}\) the full subcategory of
\(\Lambda\textup{-GMod}^{-}\) consisting of finitely cogenerated graded modules.

A graded \(\Lambda\)-module \(M\) is said to be finitely presented if there is
an exact sequence
\[
P^{-1}\longrightarrow P^0\longrightarrow M\longrightarrow 0,
\]
where both \(P^{-1}\) and \(P^0\) are finite direct sums of graded shifts of
indecomposable projective modules. We denote by
\(\Lambda\textup{-Fp}^{\mathbb Z}\) the full subcategory of
\(\Lambda\textup{-GMod}^{+}\) consisting of finitely presented graded modules.

Dually, a graded \(\Lambda\)-module \(M\) is said to be finitely copresented if
there is an exact sequence
\[
0\longrightarrow M\longrightarrow I^0\longrightarrow I^1,
\]
where both \(I^0\) and \(I^1\) are finite direct sums of graded shifts of
indecomposable injective modules. We denote by
\(\Lambda\textup{-Fcp}^{\mathbb Z}\) the full subcategory of
\(\Lambda\textup{-GMod}^{-}\) consisting of finitely copresented graded modules.

A graded \(\Lambda\)-module \(M\) is said to be perfect if it admits a finite
projective resolution by objects of \(\Lambda\textup{-Proj}^{\mathbb Z}\).
Explicitly, there is an exact sequence
\[
0\longrightarrow P^{-m}
\longrightarrow P^{-m+1}
\longrightarrow\cdots
\longrightarrow P^{-1}
\longrightarrow P^0
\longrightarrow M
\longrightarrow 0,
\]
such that each \(P^j\) is a finite direct sum of graded shifts of
indecomposable projective modules. We denote by
\(\Lambda\textup{-Pe}^{\mathbb Z}\) the full subcategory of
\(\Lambda\textup{-GMod}^{+}\) consisting of perfect graded modules.

Dually, a graded \(\Lambda\)-module \(M\) is said to be coperfect if it admits
a finite injective resolution by objects of \(\Lambda\textup{-Inj}^{\mathbb Z}\).
Explicitly, there is an exact sequence
\[
0\longrightarrow M
\longrightarrow I^0
\longrightarrow I^1
\longrightarrow\cdots
\longrightarrow I^{m}
\longrightarrow 0,
\]
such that each \(I^j\) is a finite direct sum of graded shifts of
indecomposable injective modules. We denote by
\(\Lambda\textup{-Cop}^{\mathbb Z}\) the full subcategory of
\(\Lambda\textup{-GMod}^{-}\) consisting of coperfect graded modules.
\subsection{Graded Modules over Koszul Algebras}
We next recall the definitions of Koszul algebras and their Koszul duals, and introduce the classes of graded modules that will play a central role in the subsequent development. These include Koszul and co-Koszul modules, linear and colinear modules, almost linear and almost colinear modules, and weakly Koszul and weakly co-Koszul modules.
\medskip

Recall that a complex of graded projective \(\Lambda\)-modules
\[
\cdots \longrightarrow P^{n-1}
\longrightarrow P^{n}
\longrightarrow P^{n+1}
\longrightarrow \cdots
\]
is said to be linear if, for every \(n\in\mathbb{Z}\), the term \(P^{n}\) is
generated in degree \(-n\). Equivalently, each \(P^{n}\) is a finite direct sum of
graded shifts of indecomposable projectives of the form
\[
P^{n}\cong
\bigoplus_{x\in Q_{0}} P_x\langle n\rangle^{\alpha_{x,n}},
\]
where all but finitely many of the multiplicities \(\alpha_{x,n}\) are zero.

The algebra \(\Lambda\) is called Koszul if every simple graded
\(\Lambda\)-module \(S_x\), regarded as concentrated in degree zero, admits a
linear projective resolution.

Dually, a complex of graded injective \(\Lambda\)-modules
\[
\cdots \longrightarrow I^{n-1}
\longrightarrow I^{n}
\longrightarrow I^{n+1}
\longrightarrow \cdots
\]
is said to be colinear if, for every \(n\in\mathbb{Z}\), the term \(I^{n}\) is
cogenerated in degree \(-n\). Equivalently, each \(I^{n}\) is a finite direct sum
of graded shifts of indecomposable injectives of the form
\[
I^{n}\cong
\bigoplus_{x\in Q_{0}} I_x\langle n\rangle^{\beta_{x,n}},
\]
where all but finitely many of the multiplicities \(\beta_{x,n}\) are zero.

We say that \(\Lambda\) is coKoszul if every simple graded
\(\Lambda\)-module \(S_x\), regarded as concentrated in degree zero, admits a
colinear injective resolution.

We now recall the construction of the Koszul dual algebra. Let
\(\Lambda=kQ/I\) be a Koszul algebra, and let
\[
V=(kQ)_2
\]
be the \(k\)-vector space spanned by the paths of length \(2\) in \(Q\). Set
\[
I_2=I\cap V .
\]
Let \(Q^{\mathrm{op}}\) be the opposite quiver, and put
\[
V^{\mathrm{op}}=(kQ^{\mathrm{op}})_2 .
\]
After choosing the basis of \(V\) given by the paths of length \(2\), we identify
\(V^{\mathrm{op}}\) with the dual vector space of \(V\) by means of the
corresponding dual basis. Equivalently, this gives a nondegenerate bilinear
pairing
\[
\langle -,- \rangle \colon V\times V^{\mathrm{op}}\longrightarrow k .
\]
Let
\[
I_2^{\perp}
=
\{\,u\in V^{\mathrm{op}}\mid \langle v,u\rangle=0
\text{ for all }v\in I_2\,\}
\]
be the orthogonal complement of \(I_2\) with respect to this pairing. The
Koszul dual of \(\Lambda\) is the quadratic algebra
\[
\Lambda^{!}
=
kQ^{\mathrm{op}}\big/\langle I_2^{\perp}\rangle .
\]

For further background on Koszul algebras and their Koszul duals, we refer the
reader to~\cite{8}.

\medskip 

Suppose now that \(\Lambda\) is Koszul, and let \(M\) be a graded
\(\Lambda\)-module. We say that \(M\) is Koszul if it is generated in degree zero
and admits a linear projective resolution. Dually, \(M\) is said to be coKoszul
if it is cogenerated in degree zero and admits a colinear injective
resolution.

We denote by
\( 
\Lambda\textup{-Ko}^{\mathbb{Z}}
\)
the full subcategory of \(\Lambda\textup{-GMod}^{+}\) consisting of Koszul
modules. Dually, we denote by
\( 
\Lambda\textup{-CKo}^{\mathbb{Z}}
\)
the full subcategory of \(\Lambda\textup{-GMod}^{-}\) consisting of coKoszul
modules.

These notions admit the following homological characterizations. A graded
\(\Lambda\)-module \(M\) is Koszul if and only if
\[
\operatorname{Ext}^{n}_{\Lambda\textup{-GMod}}
\bigl(M,S_x\langle -i\rangle\bigr)=0
\qquad
\text{for all }x\in Q_0\text{ and all }i\neq n .
\]
Dually, a graded \(\Lambda\)-module \(M\) is coKoszul if and only if
\[
\operatorname{Ext}^{n}_{\Lambda\textup{-GMod}}
\bigl(S_x\langle i\rangle,M\bigr)=0
\qquad
\text{for all }x\in Q_0\text{ and all }i\neq n .
\]

More generally, a graded module \(M\) is said to be \emph{linear} if there exists
an integer \(i\) such that \(M\langle i\rangle\) is Koszul. Dually, \(M\) is
said to be \emph{colinear} if there exists an integer \(i\) such that
\(M\langle i\rangle\) is coKoszul.

Let \(M\) be a graded \(\Lambda\)-module and let \(r\in\mathbb{Z}\). We define
the left truncation \(M_{\geq r}\) by
\[
(M_{\geq r})_i =
\begin{cases}
M_i, & i\geq r,\\
0, & i<r.
\end{cases}
\]
Equivalently, for each vertex \(x\in Q_0\),
\[
M_{\geq r}(x)_i =
\begin{cases}
M(x)_i, & i\geq r,\\
0, & i<r.
\end{cases}
\]
Since \(\Lambda\) is positively graded, \(M_{\geq r}\) is a graded submodule of
\(M\). We shall call \(M_{\geq r}\) the truncation of \(M\) in degrees at least
\(r\).

Dually, we define the right truncation \(M_{\leq r}\) by
\[
(M_{\leq r})_i =
\begin{cases}
M_i, & i\leq r,\\
0, & i>r.
\end{cases}
\]
Equivalently, for each vertex \(x\in Q_0\),
\[
M_{\leq r}(x)_i =
\begin{cases}
M(x)_i, & i\leq r,\\
0, & i>r.
\end{cases}
\]
In general, \(M_{\leq r}\) is not a submodule of \(M\). Rather, it is the graded
quotient
\[
M_{\leq r}=M/M_{\geq r+1}.
\]
Thus, for every \(r\in\mathbb{Z}\), one has a canonical short exact sequence
\[
0\longrightarrow M_{> r}
\longrightarrow M
\longrightarrow M_{\leq r}
\longrightarrow 0 .
\]

A graded module \(M\) is said to be \emph{almost linear} if it admits a linear
truncation, that is, if there exists an integer \(r\) such that
\(M_{\geq r}\langle r\rangle\) is Koszul. Equivalently, \(M_{\geq r}\) is
generated in degree \(r\) and admits a linear projective resolution.

Dually, \(M\) is said to be \emph{almost colinear} if it admits a colinear
truncation, that is, if there exists an integer \(r\) such that
\(M_{\leq r}\langle r\rangle\) is coKoszul. Equivalently, \(M_{\leq r}\) is
cogenerated in degree \(r\) and admits a colinear injective resolution.

We denote by
\( 
\Lambda\textup{-Al}^{\mathbb{Z}}
\)
the full subcategory of \(\Lambda\textup{-GMod}^{+}\) consisting of almost
linear graded modules, and by
\( 
\Lambda\textup{-Acl}^{\mathbb{Z}}
\)
the full subcategory of \(\Lambda\textup{-GMod}^{-}\) consisting of almost
colinear graded modules. These modules are usually referred to as modules with linear truncation and
modules with colinear truncation.
Eisenbud and Goto~\cite{25} proved that every finitely generated graded module
over a polynomial algebra admits a linear truncation. This result was later
extended by Avramov and Eisenbud~\cite{3}, who showed that every finitely
generated graded module over a commutative Noetherian Koszul algebra admits a
linear truncation. In the noncommutative setting, J{\o}rgensen~\cite{33}
obtained an analogous result for noncommutative Noetherian Koszul algebras
admitting a balanced dualizing complex.

In Section~3, we prove that every finitely presented graded module over a
quadratic monomial algebra is almost linear, and that every finitely copresented
graded module is almost colinear. As a consequence, we obtain a well-behaved
bounded Koszul duality for quadratic monomial algebras. 

\medskip 

We next recall the construction of the linear strand of a minimal graded projective resolution, together with the notions of weakly Koszul modules and absolutely Koszul algebras. We also recall the dual notions of the colinear strand of a minimal graded injective resolution, weakly coKoszul modules, and absolutely coKoszul algebras.

Let \(\Lambda\) be a graded algebra, and let \(M \in \Lambda\textup{-}\mathrm{GMod}^{+}\) be a left bounded graded module. Assume that \(M\) admits a minimal graded projective resolution whose terms are finitely generated: \[ \mathbf{P}^{\bullet}\colon \cdots \longrightarrow P^{-n} \xrightarrow{\,d^{-n}\,} P^{-n+1} \longrightarrow \cdots \longrightarrow P^{-1} \xrightarrow{\,d^{-1}\,} P^{0} \longrightarrow M \longrightarrow 0. \] Since \(\Lambda\) is graded, each differential \(d^{i}\) decomposes uniquely as \[ d^{i}=\sum_{j\geq 1} d^{i}_{j}, \] where \(d^{i}_{j}\) is homogeneous of internal degree \(j\). The \emph{linear strand} of \(\mathbf{P}^{\bullet}\), denoted by \(\operatorname{lin}(\mathbf{P}^{\bullet})\), is the complex obtained by retaining only the degree-one components of the differentials. Explicitly, \[ \operatorname{lin}(\mathbf{P}^{\bullet})\colon \cdots \longrightarrow P^{-n} \xrightarrow{\,d^{-n}_{1}\,} P^{-n+1} \longrightarrow \cdots \longrightarrow P^{-1} \xrightarrow{\,d^{-1}_{1}\,} P^{0}, \] where \(d^{i}_{\operatorname{lin}}:=d^{i}_{1}\).

Since \(\Lambda\) is presented by a quiver with homogeneous relations, each component \(d^{i}_{j}\) is induced by multiplication by linear combinations of paths of length \(j\). In particular, the differential of the linear strand is induced by multiplication by arrows of the quiver.

A graded module \(M\) is called \emph{weakly Koszul} if
\(\operatorname{lin}(\mathbf{P}^{\bullet})\) is exact.

Dually, let \(N\in \Lambda\textup{-}\mathrm{GMod}^{-}\) be a right bounded graded module admitting a minimal graded coinjective resolution whose terms are finitely cogenerated: \[ 0\longrightarrow N \longrightarrow I^{0} \xrightarrow{\delta^{0}} I^{1} \longrightarrow \cdots \longrightarrow I^{n} \xrightarrow{\delta^{n}} I^{n+1} \longrightarrow\cdots . \] Since \(\Lambda\) is graded, each differential \(\delta^{i}\) decomposes uniquely as \[ \delta^{i} = \sum_{j\geq 1}\delta^{i}_{j}, \] where \(\delta^{i}_{j}\) is homogeneous of internal degree \(j\).

The \emph{colinear strand} of \(\mathbf{I}^{\bullet}\), denoted by
\(\operatorname{colin}(\mathbf{I}^{\bullet})\), is obtained by retaining only the degree-one components of the differentials.

A graded module \(N\) is called \emph{weakly coKoszul} if \(\operatorname{colin}(\mathbf{I}^{\bullet})\) is exact. Following~\cite{31,32}, a Koszul algebra \(\Lambda\) is said to be \emph{absolutely Koszul} if every finitely generated graded \(\Lambda\)-module admits a syzygy which is weakly Koszul. Dually, a Koszul algebra \(\Lambda\) is said to be \emph{absolutely coKoszul} if every finitely cogenerated graded \(\Lambda\)-module admits a cosyzygy which is weakly coKoszul. The notion of the linear strand was introduced by Eisenbud, Fl{\o}ystad, and Schreyer in their study of resolutions over exterior algebras and the Bernstein--Gel'fand--Gel'fand correspondence; see~\cite{24}. They observed that their methods extend naturally to arbitrary Koszul dual pairs. This point of view was subsequently developed by Iwanaga, who established several fundamental connections between linearity defect, weakly Koszul modules, and Koszul duality; see~\cite{57,58}.
\subsection{Triangulated Categories}
We now collect the notions from triangulated category theory needed in the sequel, including bounded derived categories of exact categories, singularity categories, and stable categories of Gorenstein-projective modules.

\medskip 

An additive category \(\mathcal{B}\) is called an \emph{exact category} in the sense of Quillen if there exists an abelian category \(\mathcal{C}\) such that \(\mathcal{B}\) is a full extension-closed subcategory of \(\mathcal{C}\), and the exact structure on \(\mathcal{B}\) is induced by those short exact sequences in \(\mathcal{C}\) whose terms belong to \(\mathcal{B}\). We say that \(\mathcal{B}\) is \emph{weakly idempotent complete} if every retraction, equivalently every split epimorphism, in \(\mathcal{B}\) admits a kernel. Equivalently, every coretraction, equivalently every split monomorphism, in \(\mathcal{B}\) admits a cokernel. We say that \(\mathcal{B}\) satisfies the \emph{two-out-of-three property with respect to \(\mathcal{C}\)} if, for every short exact sequence \[ 0 \longrightarrow A \longrightarrow B \longrightarrow C \longrightarrow 0 \] in \(\mathcal{C}\), whenever two of the objects \(A\), \(B\), and \(C\) belong to \(\mathcal{B}\), so does the third. A full subcategory \(\mathcal{A}\subseteq\mathcal{C}\) is called a \emph{Serre subcategory} if, for every short exact sequence \[ 0 \longrightarrow A' \longrightarrow A \longrightarrow A'' \longrightarrow 0 \] in \(\mathcal{C}\), the object \(A\) belongs to \(\mathcal{A}\) if and only if both \(A'\) and \(A''\) belong to \(\mathcal{A}\). Equivalently, \(\mathcal{A}\) is closed under subobjects, quotients, and extensions in \(\mathcal{C}\). 

Let \(\mathcal{B}\) be an exact category in the above sense, with ambient abelian category \(\mathcal{C}\), and assume that \(\mathcal{B}\) satisfies the two-out-of-three property with respect to \(\mathcal{C}\). If \(\mathcal{A}\) is a Serre subcategory of \(\mathcal{B}\), then one may form the exact quotient category \[ \mathcal{B}/\mathcal{A}. \] The objects of \(\mathcal{B}/\mathcal{A}\) are the same as those of \(\mathcal{B}\), while morphisms are represented by roofs \[ X \xleftarrow{s} X' \xrightarrow{f} Y, \] where \(\ker(s)\) and \(\operatorname{coker}(s)\) belong to \(\mathcal{A}\), modulo the usual equivalence relation on such roofs.
\medskip 

A complex \(X^\bullet\in\mathsf{C}(\mathcal{B})\) is called \emph{acyclic} if it is exact when regarded as a complex in the ambient abelian category \(\mathcal{C}\).

We denote by
\( 
\mathsf{C}^{-,b}(\mathcal{B})
\)
the full subcategory of \(\mathsf{C}^{-}(\mathcal{B})\) consisting of complexes with bounded cohomology, that is,
\[
\mathsf{C}^{-,b}(\mathcal{B})
=
\left\{
X^\bullet\in\mathsf{C}^{-}(\mathcal{B})
\;\middle|\;
H^n(X^\bullet)=0
\text{ for } |n|\gg 0
\right\},
\]
and similarly,
\[
\mathsf{C}^{+,b}(\mathcal{B})
=
\left\{
X^\bullet\in\mathsf{C}^{+}(\mathcal{B})
\;\middle|\;
H^n(X^\bullet)=0
\text{ for } |n|\gg 0
\right\}.
\]

Passing to homotopy classes of chain maps yields the homotopy category
\(\mathsf{K}(\mathcal{B})\). We denote by
\(\mathsf{K}^{-}(\mathcal{B})\),
\(\mathsf{K}^{+}(\mathcal{B})\), and
\(\mathsf{K}^{b}(\mathcal{B})\)
the full triangulated subcategories consisting of complexes that are bounded above, bounded below, and bounded, respectively.

Similarly, we denote by
\( 
\mathsf{K}^{-,b}(\mathcal{B})
\,\,\text{and}\,\, 
\mathsf{K}^{+,b}(\mathcal{B})
\)
the full triangulated subcategories of
\(\mathsf{K}^{-}(\mathcal{B})\) and
\(\mathsf{K}^{+}(\mathcal{B})\),
respectively, consisting of complexes with bounded cohomology.

Let
\( 
\mathsf{Ac}^{b}(\mathcal{B})
\subseteq
\mathsf{K}^{b}(\mathcal{B})
\)
be the full subcategory of bounded acyclic complexes. Since \(\mathcal{B}\) is weakly idempotent complete, it follows that
\(\mathsf{Ac}^{b}(\mathcal{B})\)
is a thick triangulated subcategory of
\(\mathsf{K}^{b}(\mathcal{B})\).

The bounded derived category of \(\mathcal{B}\) is defined as the Verdier quotient \[ \mathsf{D}^{b}(\mathcal{B}) := \mathsf{K}^{b}(\mathcal{B}) / \mathsf{Ac}^{b}(\mathcal{B}). \] It is a triangulated category whose suspension functor is induced by the shift functor on complexes. If \(\mathcal{B}\) has enough projective objects, then there is an equivalence of triangulated categories \[ \mathsf{D}^{b}(\mathcal{B}) \;\cong\; \mathsf{K}^{-,b}\bigl(\operatorname{Proj}(\mathcal{B})\bigr), \] where \(\operatorname{Proj}(\mathcal{B})\) denotes the full subcategory of projective objects of \(\mathcal{B}\). Dually, if \(\mathcal{B}\) has enough injective objects, then there is an equivalence of triangulated categories \[ \mathsf{D}^{b}(\mathcal{B}) \;\cong\; \mathsf{K}^{+,b}\bigl(\operatorname{Inj}(\mathcal{B})\bigr), \] where \(\operatorname{Inj}(\mathcal{B})\) denotes the full subcategory of injective objects of \(\mathcal{B}\). If, moreover, \(\mathcal{B}\) has finite homological dimension and has enough projective objects, then every bounded complex admits a bounded projective resolution. Hence the preceding equivalence restricts to an equivalence \[ \mathsf{D}^{b}(\mathcal{B}) \;\cong\; \mathsf{K}^{b}\bigl(\operatorname{Proj}(\mathcal{B})\bigr). \] Dually, if \(\mathcal{B}\) has finite homological dimension and has enough injective objects, then \[ \mathsf{D}^{b}(\mathcal{B}) \;\cong\; \mathsf{K}^{b}\bigl(\operatorname{Inj}(\mathcal{B})\bigr). \]

A fundamental consequence of weak idempotent completeness is that a morphism
\( 
f\colon X^\bullet \longrightarrow Y^\bullet
\)
in \(\mathsf{D}^{b}(\mathcal{B})\) is an isomorphism if and only if its mapping cone
\(\operatorname{Cone}(f)\)
is acyclic.

Furthermore, it was shown in \cite[Proposition~2.17]{15} that if \(\mathcal{B}\) is weakly idempotent complete and \(\mathcal{A}\) is a Serre subcategory of \(\mathcal{B}\), then the quotient category
\(\mathcal{B}/\mathcal{A}\)
is again weakly idempotent complete. Consequently, the bounded derived category
\[
\mathsf{D}^{b}(\mathcal{B}/\mathcal{A})
\]
is well defined.

The above framework will be applied throughout this paper in the following setting. The category \(\mathcal{A}\) will denote the full subcategory of finite-dimensional modules over a Koszul algebra \(\Lambda\) or its Koszul dual \(\Lambda^{!}\), neither of which is assumed to be finite-dimensional. The exact category \(\mathcal{B}\) will be either the category of almost linear modules or the category of almost colinear modules. These categories are exact in the sense of Quillen, weakly idempotent complete, satisfy the two-out-of-three property with respect to the ambient abelian category \(\mathcal{C}\), and possess enough projective objects in the almost linear case and enough injective objects in the almost colinear case. Finally, \(\mathcal{C}\) will be either
\( 
\Lambda\textup{-GMod}^{+}
\,\,\text{or}\,\,
\Lambda^{!}\textup{-GMod}^{-}.
\)

Under these assumptions, \cite[Theorem~2.18]{15}, (see also \cite{46} in the
abelian case), yields a triangle equivalence
\[
\mathsf{D}^{b}(\mathcal{B})
\big/
\mathsf{D}^{b}_{\mathcal{A}}(\mathcal{B})
\;\xrightarrow{\;\sim\;}
\mathsf{D}^{b}(\mathcal{B}/\mathcal{A}),
\]
where
\[
\mathsf{D}^{b}_{\mathcal{A}}(\mathcal{B})
=
\left\{
X^\bullet\in\mathsf{D}^{b}(\mathcal{B})
\;\middle|\;
H^{n}(X^\bullet)\in\mathcal{A}
\text{ for all } n\in\mathbb Z
\right\}.
\]

We refer the reader to \cite{15,46} for further details. The above equivalence plays a fundamental role in the formulation of graded singular Koszul duality and generalized Bernstein--Gelfand--Gelfand correspondences developed in \cite{15}.

\medskip

Assume now that \(\mathcal{B}\) has enough projective objects. The \emph{singularity category} of \(\mathcal{B}\) is defined to be the Verdier quotient
\[
\mathsf{D}_{\mathrm{sg}}(\mathcal{B})
:=
\mathsf{D}^{b}(\mathcal{B})
\big/
\mathsf{K}^{b}\bigl(\operatorname{Proj}(\mathcal{B})\bigr),
\]

The singularity category measures the difference between the bounded derived category and the homotopy category of bounded complexes of projective objects. In particular,
\[
\mathsf{D}_{\mathrm{sg}}(\mathcal{B})=0
\]
whenever every object of \(\mathcal{B}\) has finite projective dimension. 

\medskip 

Assume that \(\Lambda\) is a left coherent graded algebra. A finitely presented graded \(\Lambda\)-module \(M\) is called \emph{Gorenstein projective} if there exists an exact complex
\[
P^\bullet=
\bigl(
\cdots
\longrightarrow
P^{-1}
\longrightarrow
P^{0}
\longrightarrow
P^{1}
\longrightarrow
\cdots
\bigr)
\]
of finitely generated graded projective \(\Lambda\)-modules such that
\( 
M \cong \ker(P^{0}\to P^{1}),
\)
and such that the complex
\( 
\operatorname{Hom}_{\Lambda}(P^\bullet,Q)
\)
is exact for every finitely generated graded projective \(\Lambda\)-module \(Q\).

We denote by
\( 
\Lambda\textup{-GProj}^{\mathbb Z}
\)
the full subcategory of finitely presented graded Gorenstein-projective \(\Lambda\)-modules.

Its stable category will be denoted by
\( 
\Lambda\textup{-}\underline{\mathrm{GProj}}^{\mathbb Z}.
\)

A fundamental theorem of Buchweitz~\cite{20} and, independently, Orlov~\cite{48}, see also~\cite{21} for a generalization to coherent algebras, asserts that if \(\Lambda\) is Iwanaga--Gorenstein, then there is an equivalence of triangulated categories \[ \Lambda\textup{-}\underline{\mathrm{GProj}}^{\mathbb Z} \;\xrightarrow{\;\sim\;} \mathsf{D}_{\mathrm{sg}} \bigl(\Lambda\textup{-}\mathrm{Fp}^{\mathbb Z}\bigr), \] where \[ \mathsf{D}_{\mathrm{sg}} \bigl(\Lambda\textup{-}\mathrm{Fp}^{\mathbb Z}\bigr) := \mathsf{D}^{b} \bigl(\Lambda\textup{-}\mathrm{Fp}^{\mathbb Z}\bigr) \Big/ \mathsf{K}^{b} \bigl(\Lambda\textup{-}\mathrm{Proj}^{\mathbb Z}\bigr) \] denotes the graded singularity category of \(\Lambda\).

\medskip 

If \(\Lambda\) is left graded coherent, we define the tails category of
\(\Lambda\) by
\[
\operatorname{qgr}(\Lambda)
=
\Lambda\textup{-}\mathrm{Fp}^{\mathbb{Z}}
\big/
\Lambda\textup{-}\mathrm{gmod}.
\]
Dually, if \(\Lambda\) is left graded cocoherent, we define the cotails
category of \(\Lambda\) by
\[
\operatorname{cqgr}(\Lambda)
=
\Lambda\textup{-}\mathrm{Fcp}^{\mathbb{Z}}
\big/
\Lambda\textup{-}\mathrm{gmod}.
\]
In algebraic geometry, the tails category is usually considered under the
stronger assumption that \(\Lambda\) is left graded noetherian.

\medskip 

We end this subsection by recalling a triangulated construction naturally associated with stable categories. Let \(\mathcal D\) be a weakly idempotent complete exact category, in the sense of Quillen, and assume that \(\mathcal D\) has enough injectives. We denote by
\[
\overline{\mathcal D} = \mathcal D/\operatorname{Inj}(\mathcal D)
\]
the stable category obtained by factoring out morphisms which factor through injective objects. In general, the category \(\overline{\mathcal D}\) is not triangulated, but it carries a natural right triangulated structure. One can nevertheless associate to it a triangulated category, called the \emph{stabilization} of \(\overline{\mathcal D}\), and denoted by
\[
\mathsf S(\overline{\mathcal D}).
\]
Its objects may be represented by pairs \((X,n)\), where \(X\in\overline{\mathcal D}\) and \(n\in\mathbb Z\), and its morphism spaces are given by
\[
\operatorname{Hom}_{\mathsf{S}(\overline{\mathcal D})}
\bigl((p,X),(q,Y)\bigr)
=
\varinjlim_{r\geq p,q}
\operatorname{Hom}_{\overline{\mathcal D}}
\bigl(\Sigma^{r-p}X,\Sigma^{r-q}Y\bigr),
\]
Here \(\Sigma\) denotes the suspension functor on \(\overline{\mathcal D}\), defined by choosing, for each \(X\in\mathcal D\), a conflation
\[
0\longrightarrow X\longrightarrow I\longrightarrow \Sigma X\longrightarrow 0
\]
with \(I\) injective. We refer to \cite{9,38} for details and a systematic study of this construction. The stabilization will be used below in the formulation of a generalized BGG correspondence, or equivalently, a Koszul-duality description of the bounded derived category of coherent sheaves. A basic property of the stabilization is that there is a triangulated equivalence
\[
\mathsf{S}(\overline{\mathcal D})
\simeq
\mathsf{D}^{b}(\mathcal D)
\big/
\mathsf{K}^{b}\bigl(\operatorname{Inj}(\mathcal D)\bigr)
\qquad\text{(see \cite{38}).}
\]

\subsection{Koszul Duality}
We conclude the preliminary section by reviewing the construction of Koszul duality established in~\cite{8, 18,19,44}. Since only a special case will be needed in the sequel, we restrict our attention to the categories
\[
\Lambda^{!}\textup{-GMod}^{-}
\qquad\text{and}\qquad
\Lambda\textup{-GMod}^{+}.
\]

\medskip

Following \cite{15,18, 19}, we define the \emph{Koszul functor}
\[
K \colon
\Lambda^{!}\textup{-GMod}^{-}
\longrightarrow
\mathsf{C}\bigl(\Lambda\textup{-GMod}^{+}\bigr).
\]

For a module
\(M\in \Lambda^{!}\textup{-GMod}^{-}\),
the complex \(K(M)\) is given by
\[
K(M)^n
=
\bigoplus_{x\in Q_0}
P_x\langle n\rangle
\otimes_k
M_n(x),
\]
with differential
\[
d^n
=
\sum_{\alpha:y\to x}
P_\alpha
\otimes
M(\alpha^{\operatorname{op}}),
\qquad
P_\alpha(p)=p\alpha.
\]

Dually, one defines the coKoszul functor 
\[
K^{\vee} \colon
\Lambda\textup{-GMod}^{+}
\longrightarrow
\mathsf{C}\bigl(\Lambda^{!}\textup{-GMod}^{-}\bigr),
\]
 For a module
\(M\in \Lambda\textup{-GMod}^{+}\),
the complex \(K^{\vee}(M)\) is given by
\[
K^{\vee}(M)^n
=
\bigoplus_{x\in Q_0}
I_x^{!}\langle n\rangle
\otimes_k
M_n(x).
\]

We now introduce the full subcategories
\[
\mathsf{C}^{\downarrow}(\Lambda^{!}\textup{-GMod}^{-})
\subseteq
\mathsf{C}(\Lambda^{!}\textup{-GMod}^{-})
\]
and
\[
\mathsf{C}^{\uparrow}(\Lambda\textup{-GMod}^{+})
\subseteq
\mathsf{C}(\Lambda\textup{-GMod}^{+}),
\]
which play a central role in Koszul duality.

A complex \(X^\bullet\) belongs to
\(\mathsf{C}^{\downarrow}(\Lambda^{!}\textup{-GMod}^{-})\)
if
\[
X^i_j=0
\qquad
\text{whenever }
i+j\gg 0
\quad\text{or}\quad
i\ll 0.
\]

Similarly, a complex \(Y^\bullet\) belongs to
\(\mathsf{C}^{\uparrow}(\Lambda\textup{-GMod}^{+})\)
if
\[
Y^i_j=0
\qquad
\text{whenever }
i+j\ll 0
\quad\text{or}\quad
i\gg 0.
\]

Here \(i\) denotes the cohomological degree and \(j\) the internal grading.

\medskip

The corresponding homotopy categories are denoted by
\[
\mathsf{K}^{\downarrow}(\Lambda^{!}\textup{-GMod}^{-})
\qquad\text{and}\qquad
\mathsf{K}^{\uparrow}(\Lambda\textup{-GMod}^{+}),
\]
while the corresponding derived categories are denoted by
\[
\mathsf{D}^{\downarrow}(\Lambda^{!}\textup{-GMod}^{-})
\qquad\text{and}\qquad
\mathsf{D}^{\uparrow}(\Lambda\textup{-GMod}^{+}).
\]

As shown in \cite[Section~5]{18, 19}, the functor \(K\) extends to complexes by applying \(K\) degreewise and then taking the totalization of the resulting double complex. This yields a functor
\[
\mathcal{F}\colon
\mathsf{C}^{\downarrow}(\Lambda^{!}\textup{-GMod}^{-})
\longrightarrow
\mathsf{C}^{\uparrow}(\Lambda\textup{-GMod}^{+}).
\]

Since \(\mathcal{F}\) preserves homotopies and acyclic complexes, it induces functors
\[
\mathscr{F}\colon
\mathsf{K}^{\downarrow}(\Lambda^{!}\textup{-GMod}^{-})
\longrightarrow
\mathsf{K}^{\uparrow}(\Lambda\textup{-GMod}^{+})
\]
and
\[
\mathfrak{F}\colon
\mathsf{D}^{\downarrow}(\Lambda^{!}\textup{-GMod}^{-})
\longrightarrow
\mathsf{D}^{\uparrow}(\Lambda\textup{-GMod}^{+}).
\]

Dually, the coKoszul functor
\[
K^{\vee}\colon
\Lambda\textup{-GMod}^{+}
\longrightarrow
\mathsf{C}\bigl(\Lambda^{!}\textup{-GMod}^{-}\bigr)
\]
extends to functors
\[
\mathcal{G}\colon
\mathsf{C}^{\uparrow}(\Lambda\textup{-GMod}^{+})
\longrightarrow
\mathsf{C}^{\downarrow}(\Lambda^{!}\textup{-GMod}^{-}),
\]
\[
\mathscr{G}\colon
\mathsf{K}^{\uparrow}(\Lambda\textup{-GMod}^{+})
\longrightarrow
\mathsf{K}^{\downarrow}(\Lambda^{!}\textup{-GMod}^{-}),
\]
and
\[
\mathfrak{G}\colon
\mathsf{D}^{\uparrow}(\Lambda\textup{-GMod}^{+})
\longrightarrow
\mathsf{D}^{\downarrow}(\Lambda^{!}\textup{-GMod}^{-}).
\]

The following theorem is the fundamental result of Koszul duality. It was established in greater generality in \cite{8, 18,19, 44}; in our setting, it takes the following form.

\begin{theorem}[{\cite[Theorem~5.16]{18}}, {\cite[Theorem~5.7]{19}}]
\label{thm:koszul-duality}
The functors \(\mathfrak{F}\) and \(\mathfrak{G}\) are quasi-inverse
equivalences of triangulated categories. In particular, there is an equivalence
of triangulated categories
\[
\begin{tikzcd}[column sep=huge]
\mathsf{D}^{\downarrow}\!\bigl(\Lambda^{!}\textup{-}\mathrm{GMod}^{-}\bigr)
\arrow[r, bend left=35, "\mathfrak{F}"]
&
\mathsf{D}^{\uparrow}\!\bigl(\Lambda\textup{-}\mathrm{GMod}^{+}\bigr)
\arrow[l, bend left=35, "\mathfrak{G}"']
\end{tikzcd}
\]
which we refer to as the \emph{Koszul duality}.
\end{theorem}

The next proposition records some basic properties of the Koszul duality functors.

\begin{proposition}[{\cite[Lemma~5.10]{18}}, {\cite[Lemma~4.2]{15}}]
The functor \(\mathfrak{F}\) takes the injective module 
\(I^{!}_{x}\langle i\rangle\) 
to the minimal projective resolution of the simple module 
\(S_{x}\langle -i\rangle\), 
and the simple module 
\(S^{!}_{x}\langle i\rangle\) 
to the projective module 
\(P_{x}\langle -i\rangle[i]\).

Dually, the functor \(\mathfrak{G}\) takes the projective module 
\(P_{x}\langle i\rangle\) 
to the minimal injective resolution of the simple module 
\(S^{!}_{x}\langle -i\rangle\), 
and the simple module 
\(S_{x}\langle i\rangle\) 
to the injective module 
\(I^{!}_{x}\langle -i\rangle[i]\).

More generally, for every 
\( 
X^\bullet \in 
\mathsf{D}^{\downarrow} 
(\Lambda^{!}\textup{-GMod}^{-}) 
\) 
and every \(i\in\mathbb Z\), one has
\[ 
\mathfrak{F}\bigl(X^\bullet\langle i\rangle\bigr) 
\cong 
\mathfrak{F}(X^\bullet)\langle -i\rangle[i]. 
\]
See also \cite{8, 44}.
\end{proposition}
\section{General Bounded Koszul Duality}
In this section, we establish our main results on infinite-dimensional Koszul algebras. More precisely, we prove a bounded Koszul duality for Koszul algebras without imposing any additional finiteness assumptions on either the algebra \(\Lambda\) or its Koszul dual \(\Lambda^{!}\). This extends the results of \cite{15} from the finite-dimensional setting to a broad class of infinite-dimensional Koszul algebras. As applications, we obtain a graded singular Koszul duality and a generalized Bernstein--Gelfand--Gelfand correspondence for coherent Iwanaga--Gorenstein Koszul algebras.

We then specialize our results to two important classes of Koszul algebras, namely, quadratic monomial algebras, and absolutely Koszul algebras with an extra homological property.
\subsection{Bounded Koszul Duality for Infinite-Dimensional Koszul Algebras}
Before stating and proving the main theorem of this subsection, we establish several preliminary results. The first proposition clarifies the relationship between coKoszul modules over \(\Lambda^{!}\) and Koszul modules over \(\Lambda\).

\begin{proposition}
The Koszul duality restricts to an equivalence of categories
\[
\begin{tikzcd}[column sep=huge]
\Lambda^{!}\textup{-CKo}^{\mathbb{Z}}
\arrow[r, bend left=35, "\mathfrak{F}"]
&
\Lambda\textup{-Ko}^{\mathbb{Z}}
\arrow[l, bend left=35, "\mathfrak{G}"']
\end{tikzcd}
\]
\end{proposition}
\begin{proof}
We shall prove only one direction, since the other is dual.

Let \(M\) be a coKoszul module over \(\Lambda^{!}\). We shall show that the complex \(K(M)\) is exact except in cohomological degree \(0\).

It is well known that
\[
H^{n}(K(M))
\cong
\bigoplus_{i\in\mathbb Z}
\operatorname{Hom}_{\mathsf D(\Lambda\textup{-GMod})}
\bigl(\Lambda,K(M)\langle i\rangle[n]\bigr).
\]

Hence
\[
H^{n}(K(M))
\cong
\bigoplus_{i\in\mathbb Z}
\operatorname{Hom}_{\mathsf D(\Lambda\textup{-GMod})}
\Bigl(
\bigoplus_{\substack{x\in Q_{0}\\ j\in\mathbb Z}}
P_{x}\langle j\rangle,
K(M)\langle i\rangle[n]
\Bigr).
\]

By Proposition~2.2, we obtain
\[
\begin{aligned}
H^{n}(K(M))
&\cong
\bigoplus_{i\in\mathbb Z}
\operatorname{Hom}_{\mathsf D(\Lambda^{!}\textup{-GMod})}
\Bigl(
\bigoplus_{\substack{x\in Q_{0}\\ j\in\mathbb Z}}
S^{!}_{x}\langle -j\rangle[j],
M\langle -i\rangle[n+i]
\Bigr)
\\
&\cong
\bigoplus_{i\in\mathbb Z}
\prod_{\substack{x\in Q_{0}\\ j\in\mathbb Z}}
\operatorname{Hom}_{\mathsf D(\Lambda^{!}\textup{-GMod})}
\Bigl(
S^{!}_{x}\langle i-j\rangle,
M[n+i-j]
\Bigr)
\\
&\cong
\bigoplus_{i\in\mathbb Z}
\prod_{\substack{x\in Q_{0}\\ j\in\mathbb Z}}
\operatorname{Ext}^{\,n+i-j}_{\Lambda^{!}\textup{-GMod}}
\bigl(S^{!}_{x}\langle i-j\rangle,M\bigr).
\end{aligned}
\]

Since \(M\) is coKoszul, we have
\[
\operatorname{Ext}^{\,n+i-j}_{\Lambda^{!}\textup{-GMod}}
\bigl(S^{!}_{x}\langle i-j\rangle,M\bigr)
=
0
\]
whenever
\[
n+i-j\neq i-j,
\]
that is, whenever \(n\neq 0\). Therefore,
\[
H^{n}(K(M))=0
\qquad
\text{for all } n\neq 0.
\]

Hence \(K(M)\) has cohomology concentrated in degree \(0\). It follows that \(K(M)\) is isomorphic in
\(\mathsf{D}^{\uparrow}(\Lambda\textup{-GMod}^{+})\)
to its degree-zero cohomology module \(H^{0}(K(M))\). Since \(K(M)\) is a linear complex, \(H^{0}(K(M))\) is a Koszul \(\Lambda\)-module. Therefore,
\[
H^{0}(K(M))
\in
\Lambda\textup{-Ko}^{\mathbb Z}.
\]

\medskip 

This proves that \(\mathfrak{F}\) sends coKoszul \(\Lambda^{!}\)-modules to Koszul \(\Lambda\)-modules.
\end{proof}
Next, we use Proposition~3.1 to characterize certain subcategories of modules over \(\Lambda\) and its Koszul dual \(\Lambda^{!}\), which will play a central role in the formulation of our bounded  derived Koszul duality. To this end, we first restrict Koszul duality to suitable bounded derived categories.

Following the construction of Koszul duality, we consider the Koszul functors
\[
K \colon
\Lambda^{!}\textup{-GMod}^{-}
\longrightarrow
\mathsf{C}\bigl(\Lambda\textup{-GMod}^{+}\bigr)
\]
and
\[
K^{\vee} \colon
\Lambda\textup{-GMod}^{+}
\longrightarrow
\mathsf{C}\bigl(\Lambda^{!}\textup{-GMod}^{-}\bigr).
\]

We denote by
\(\Lambda^{!}\textup{-}\mathrm{GMod}^{-,b}\)
the full subcategory of
\(\Lambda^{!}\textup{-}\mathrm{GMod}^{-}\)
formed by those modules \(M\) for which the complex \(K(M)\) has bounded
cohomology. Dually, we denote by
\(\Lambda\textup{-}\mathrm{GMod}^{+,b}\)
the full subcategory of
\(\Lambda\textup{-}\mathrm{GMod}^{+}\)
formed by those modules \(N\) for which the complex \(K^{\vee}(N)\) has bounded
cohomology.

These categories are weakly idempotent complete and satisfy, respectively, two-out-of-three  property  inside
\(\Lambda^{!}\textup{-}\mathrm{GMod}^{-}\) and
\(\Lambda\textup{-}\mathrm{GMod}^{+}\). Furthermore,
\(\Lambda^{!}\textup{-}\mathrm{gmod}\) is a Serre subcategory of
\(\Lambda^{!}\textup{-}\mathrm{GMod}^{-,b}\), and
\(\Lambda\textup{-}\mathrm{gmod}\) is a Serre subcategory of
\(\Lambda\textup{-}\mathrm{GMod}^{+,b}\).

\medskip 

By the construction of Koszul duality; see, for example, \cite[Theorem~12]{39}, the Koszul functors induce equivalences
\[
K \colon
\Lambda^{!}\textup{-GMod}^{-}
\xrightarrow{\ \sim\ }
\mathcal{LC}^{-}\bigl(\Lambda\textup{-Proj}^{\mathbb Z}\bigr)
\]
and
\[
K^{\vee} \colon
\Lambda\textup{-GMod}^{+}
\xrightarrow{\ \sim\ }
\mathcal{CLC}^{+}\bigl(\Lambda^{!}\textup{-Inj}^{\mathbb Z}\bigr).
\]

Here
\( 
\mathcal{LC}^{-}\bigl(\Lambda\textup{-Proj}^{\mathbb Z}\bigr)
\)
denotes the category of linear complexes of graded projective \(\Lambda\)-modules, that is, complexes whose \(n\)-th term is of the form
\[
K(M)^n
=
\bigoplus_{x\in Q_0}
P_x\langle n\rangle
\otimes_k
M_n(x).
\]
Similarly,
\( 
\mathcal{CLC}^{+}\bigl(\Lambda^{!}\textup{-Inj}^{\mathbb Z}\bigr)
\)
denotes the category of colinear complexes of graded injective \(\Lambda^{!}\)-modules, that is, complexes whose \(n\)-th term is of the form
\[
K^{\vee}(N)^n
=
\bigoplus_{x\in Q_0}
I_x^{!}\langle n\rangle
\otimes_k
N_n(x).
\]

By definition of the subcategories
\(\Lambda^{!}\textup{-GMod}^{-,b}\)
and
\(\Lambda\textup{-GMod}^{+,b}\),
the above equivalences restrict to equivalences
\[
K \colon
\Lambda^{!}\textup{-GMod}^{-,b}
\xrightarrow{\ \sim\ }
\mathcal{LC}^{-,b}\bigl(\Lambda\textup{-Proj}^{\mathbb Z}\bigr)
\]
and
\[
K^{\vee} \colon
\Lambda\textup{-GMod}^{+,b}
\xrightarrow{\ \sim\ }
\mathcal{CLC}^{+,b}\bigl(\Lambda^{!}\textup{-Inj}^{\mathbb Z}\bigr).
\]

The significance of these categories is illustrated by the finite-dimensional case. Indeed, it was shown in \cite[Proposition~3.4]{15} that, when \(\Lambda\) is finite-dimensional, one has
\( 
\Lambda^{!}\textup{-GMod}^{-,b}
=
\Lambda^{!}\textup{-Cop}^{\mathbb Z}.
\)
Dually,
\( 
\Lambda^{!}\textup{-GMod}^{+,b}
=
\Lambda^{!}\textup{-Pe}^{\mathbb Z}.
\)
Thus,
\(\Lambda^{!}\textup{-GMod}^{-,b}\)
and
\(\Lambda\textup{-GMod}^{+,b}\)
may be viewed as natural generalizations of the categories of finitely copresented and finitely presented graded modules, respectively.

\medskip 

The Koszul functor
\[
K \colon
\Lambda^{!}\textup{-GMod}^{-,b}
\xrightarrow{\ \sim\ }
\mathcal{LC}^{-,b}\bigl(\Lambda\textup{-Proj}^{\mathbb Z}\bigr)
\]
admits a natural extension to bounded complexes.

Applying \(K\) degreewise defines a \(k\)-linear functor
\[
\mathfrak{K}\colon
\mathsf{C}^{b}\!\bigl(\Lambda^{!}\textup{-}\mathrm{GMod}^{-,b}\bigr)
\longrightarrow
\mathsf{DC}^{b,-}\!\bigl(\Lambda\textup{-}\mathrm{Proj}^{\mathbb Z}\bigr),
\]
where
\(\mathsf{DC}^{b,-}\!\bigl(\Lambda\textup{-}\mathrm{Proj}^{\mathbb Z}\bigr)\)
denotes the category of double cochain complexes with terms in
\(\Lambda\textup{-}\mathrm{Proj}^{\mathbb Z}\), bounded in the horizontal
direction and right bounded in the vertical direction.

Let
\[
M^\bullet=
\bigl(
\cdots \longrightarrow M^{p-1}
\xrightarrow{\,d^{p-1}\,}
M^{p}
\xrightarrow{\,d^{p}\,}
M^{p+1}
\longrightarrow \cdots
\bigr)
\]
be a bounded cochain complex with
\(M^{p}\in \Lambda^{!}\textup{-}\mathrm{GMod}^{-,b}\).
For \(x\in Q_0\) and \(q\in\mathbb Z\), let
\(M^p_{x,q}\) denote the homogeneous component of internal degree \(q\) at
the vertex \(x\).

The associated double complex \(\mathfrak K(M^\bullet)\) has terms
\[
\mathfrak K(M^\bullet)^{p,q}
=
\bigoplus_{x\in Q_0}
P_x\langle q\rangle\otimes_k M^p_{x,q}.
\]

Its horizontal differential is induced by the differential of
\(M^\bullet\), whereas its vertical differential is induced by the action of
the arrows of \(Q^{\mathrm{op}}\). A typical portion of
\(\mathfrak K(M^\bullet)\) is therefore of the form
\[
\begin{array}{cccccc}
\cdots &
\longrightarrow &
\displaystyle\bigoplus_{y\in Q_0}
P_y\langle q+1\rangle\otimes_k M^p_{y,q+1}
&
\xrightarrow{\ d^{p,q+1}_1\ }
&
\displaystyle\bigoplus_{y\in Q_0}
P_y\langle q+1\rangle\otimes_k M^{p+1}_{y,q+1}
&
\longrightarrow \cdots
\\[1.4em]
&&
\uparrow\, d^{p,q}_2
&&
\uparrow\, d^{p+1,q}_2
&
\\[-0.2em]
\cdots &
\longrightarrow &
\displaystyle\bigoplus_{x\in Q_0}
P_x\langle q\rangle\otimes_k M^p_{x,q}
&
\xrightarrow{\ d^{p,q}_1\ }
&
\displaystyle\bigoplus_{x\in Q_0}
P_x\langle q\rangle\otimes_k M^{p+1}_{x,q}
&
\longrightarrow \cdots .
\end{array}
\]
The horizontal differential
\[
d^{p,q}_{1}\colon
\bigoplus_{x\in Q_{0}}
P_x\langle q\rangle\otimes M^p_{x,q}
\longrightarrow
\bigoplus_{x\in Q_{0}}
P_x\langle q\rangle\otimes M^{p+1}_{x,q}
\]
is induced by the differential of the complex \(M^\bullet\). Explicitly,
\[
d^{p,q}_{1}
=
\sum_{x\in Q_{0}}
\mathrm{id}_{P_x\langle q\rangle}
\otimes d^p_{x,q},
\]
where
\[
d^p_{x,q}\colon
M^p_{x,q}\longrightarrow M^{p+1}_{x,q}
\]
denotes the \((x,q)\)-component of the differential of \(M^\bullet\).

The vertical differential
\[
d^{p,q}_{2}\colon
\bigoplus_{x\in Q_{0}}
P_x\langle q\rangle\otimes M^p_{x,q}
\longrightarrow
\bigoplus_{y\in Q_{0}}
P_y\langle q+1\rangle\otimes M^p_{y,q+1}
\]
is determined by its \((y,x)\)-component
\[
(d^{p,q}_{2})_{yx}
=
\sum_{\alpha:y\to x}
P_\alpha\otimes M^p(\alpha^{\mathrm{op}}),
\]
where the sum ranges over all arrows \(\alpha\) of \(Q\) with
source \(y\) and target \(x\). Here
\[
P_\alpha\colon
P_x\langle q\rangle
\longrightarrow
P_y\langle q+1\rangle
\]
is the morphism induced by right multiplication by \(\alpha\),
namely \(P_\alpha(u)=u\alpha\), and
\[
M^p(\alpha^{\mathrm{op}})\colon
M^p_{x,q}
\longrightarrow
M^p_{y,q+1}
\]
is the corresponding structure morphism of the graded
\(\Lambda^!\)-module \(M^p\).

The identities
\[
d_1^2=0,
\qquad
d_2^2=0,
\qquad
d_1d_2=d_2d_1
\]
follow immediately from the fact that \(M^\bullet\) is a complex and from
the definition of the Koszul functor. Consequently,
\(\mathfrak K(M^\bullet)\) is a double cochain complex.

The double complex \(\mathfrak{K}(M^\bullet)\) gives rise to a total complex in
the usual way. More generally, let
\[
A=(A^{p,q},d_1,d_2)
\in
\mathsf{DC}^{b,-}\!\bigl(\Lambda\textup{-}\mathrm{Proj}^{\mathbb{Z}}\bigr).
\]
Its totalization is the complex
\[
\operatorname{Tot}(A)^n
=
\bigoplus_{p+q=n}A^{p,q},
\]
equipped with differential
\[
d_{\operatorname{Tot}}
=
d_1+(-1)^p d_2.
\]
Since \(d_1^2=d_2^2=0\) and \(d_1d_2=d_2d_1\), the differential
\(d_{\operatorname{Tot}}\) squares to zero. Consequently, totalization defines
a \(k\)-linear functor
\[
\operatorname{Tot}\colon
\mathsf{DC}^{b,-}\!\bigl(\Lambda\textup{-}\mathrm{Proj}^{\mathbb{Z}}\bigr)
\longrightarrow
\mathsf{C}^{-,b}\!\bigl(\Lambda\textup{-}\mathrm{Proj}^{\mathbb{Z}}\bigr).
\]

Composing \(\mathfrak{K}\) with \(\operatorname{Tot}\), we obtain a
\(k\)-linear functor
\[
\mathcal{F}
:=
\operatorname{Tot}\circ\mathfrak{K}
\colon
\mathsf{C}^{b}\!\bigl(\Lambda^{!}\textup{-}\mathrm{GMod}^{-,b}\bigr)
\longrightarrow
\mathsf{C}^{-,b}\!\bigl(\Lambda\textup{-}\mathrm{Proj}^{\mathbb{Z}}\bigr).
\]
 By construction,
\(\mathcal{F}\) extends the Koszul functor \(K\). Indeed, if
\(M\in\Lambda^{!}\textup{-}\mathrm{GMod}^{-,b}\) is viewed as a complex
concentrated in degree \(0\), then
\[
\mathcal{F}(M)=K(M).
\]
Accordingly, \(\mathcal{F}\) restricts to
\[
K\colon
\Lambda^{!}\textup{-}\mathrm{GMod}^{-,b}
\longrightarrow
\mathcal{LC}^{-,b}\!\bigl(\Lambda\textup{-}\mathrm{Proj}^{\mathbb{Z}}\bigr).
\]
The functor \(\mathcal{F}\) is compatible with homotopies, and therefore induces a \(k\)-linear triangulated functor
\[
\mathscr{F}\colon
\mathsf{K}^{b}\!\bigl(\Lambda^{!}\textup{-}\mathrm{GMod}^{-,b}\bigr)
\longrightarrow
\mathsf{K}^{-,b}\!\bigl(\Lambda\textup{-}\mathrm{Proj}^{\mathbb{Z}}\bigr).
\]
Since the category \(\Lambda\textup{-}\mathrm{GMod}^{+,b}\) has enough
projectives, every object of
\(\mathsf{D}^{b}\!\bigl(\Lambda\textup{-}\mathrm{GMod}^{+,b}\bigr)\)
admits a right bounded projective resolution with bounded cohomology. Hence
the canonical localization functor induces an equivalence
\[
\mathfrak{T}\colon
\mathsf{K}^{-,b}\!\bigl(\Lambda\textup{-}\mathrm{Proj}^{\mathbb{Z}}\bigr)
\xrightarrow{\ \sim\ }
\mathsf{D}^{b}\!\bigl(\Lambda\textup{-}\mathrm{GMod}^{+,b}\bigr).
\]
Moreover, \(\mathscr{F}\) sends acyclic complexes to acyclic complexes.

\medskip 

Hence
\(\mathscr{F}\) factors through the derived categories and induces a triangulated functor
\[
\mathfrak{F}\colon
\mathsf{D}^{b}\!\bigl(\Lambda^{!}\textup{-}\mathrm{GMod}^{-,b}\bigr)
\longrightarrow
\mathsf{D}^{b}\!\bigl(\Lambda\textup{-}\mathrm{GMod}^{+,b}\bigr).
\]

Thus the derived Koszul functor is given by
\[
\mathfrak{F}
=
\mathfrak{T}\circ\mathscr{F}.
\]

Equivalently, the construction is summarized by the commutative diagram of triangulated categories

\[
\xymatrix@C=5.5em@R=4em{
\mathsf{C}^{b}\!\bigl(\Lambda^{!}\textup{-}\mathrm{GMod}^{-,b}\bigr)
\ar[r]^-{\mathcal{F}} \ar[d]_{\scriptstyle \pi} &
\mathsf{C}^{-,b}\!\bigl(\Lambda\textup{-}\mathrm{Proj}^{\mathbb{Z}}\bigr)
\ar[d]^{\scriptstyle \pi'} \\
\mathsf{K}^{b}\!\bigl(\Lambda^{!}\textup{-}\mathrm{GMod}^{-,b}\bigr)
\ar[r]^-{\mathscr{F}} \ar[d]_{\scriptstyle \epsilon} &
\mathsf{K}^{-,b}\!\bigl(\Lambda\textup{-}\mathrm{Proj}^{\mathbb{Z}}\bigr)
\ar[d]^{\scriptstyle \mathfrak{T}} \\
\mathsf{D}^{b}\!\bigl(\Lambda^{!}\textup{-}\mathrm{GMod}^{-,b}\bigr)
\ar[r]^-{\mathfrak{F}} &
\mathsf{D}^{b}\!\bigl(\Lambda\textup{-}\mathrm{GMod}^{+,b}\bigr)
}
\]
Similarly, one constructs the dual  functor
\[
\mathfrak{G}\colon
\mathsf{D}^{b}\!\bigl(\Lambda\textup{-}\mathrm{GMod}^{+,b}\bigr)
\longrightarrow
\mathsf{D}^{b}\!\bigl(\Lambda^{!}\textup{-}\mathrm{GMod}^{-,b}\bigr).
\]

Note that by abuse of notation we denoted the induced functors by $\mathfrak{F}$ and $\mathfrak{G}$ even though they are slightly different from the Koszul duality stated in theorem 2.1.

\medskip 

Since \(\Lambda^{!}\textup{-}\mathrm{GMod}^{-,b}\) has enough injectives and
\(\Lambda\textup{-}\mathrm{GMod}^{+,b}\) has enough projectives, it follows
from \cite[Proposition~2.14]{15} that
\(\mathsf{D}^{b}\!\bigl(\Lambda^{!}\textup{-}\mathrm{GMod}^{-,b}\bigr)\)
 is a full triangulated subcategory of
\(\mathsf{D}^{\downarrow}\!\bigl(\Lambda^{!}\textup{-}\mathrm{GMod}^{-}\bigr)\),
and that
\(\mathsf{D}^{b}\!\bigl(\Lambda\textup{-}\mathrm{GMod}^{+,b}\bigr)\)
is a full triangulated subcategory of
\(\mathsf{D}^{\uparrow}\!\bigl(\Lambda\textup{-}\mathrm{GMod}^{+}\bigr)\).

\medskip 
We are now ready to prove the main result of this section, which we refer to
as the bounded derived Koszul duality. It extends
\cite[Theorem~3.8]{15} to infinite-dimensional Koszul algebras.

\begin{theorem}
The Koszul duality functors induce mutually inverse triangulated equivalences
\[
\begin{tikzcd}[column sep=huge]
\mathsf{D}^{b}\!\bigl(\Lambda^{!}\textup{-}\mathrm{Acl}^{\mathbb{Z}}\bigr)
\arrow[r, bend left=35, "\mathfrak{F}"]
&
\mathsf{D}^{b}\!\bigl(\Lambda\textup{-}\mathrm{Al}^{\mathbb{Z}}\bigr)
\arrow[l, bend left=35, "\mathfrak{G}"']
\end{tikzcd}
\]
Moreover, as triangulated categories,
\(\mathsf{D}^{b}\!\bigl(\Lambda^{!}\textup{-}\mathrm{Acl}^{\mathbb{Z}}\bigr)\)
is generated by colinear \(\Lambda^{!}\)-modules, and
\(\mathsf{D}^{b}\!\bigl(\Lambda\textup{-}\mathrm{Al}^{\mathbb{Z}}\bigr)\)
is generated by linear \(\Lambda\)-modules.
\end{theorem}

\begin{proof}
By Theorem~2.1, the Koszul duality functors restrict to mutually inverse
triangulated equivalences
\[
\begin{tikzcd}[column sep=huge]
\mathsf{D}^{b}\!\bigl(\Lambda^{!}\textup{-}\mathrm{GMod}^{-,b}\bigr)
\arrow[r, bend left=35, "\mathfrak{F}"]
&
\mathsf{D}^{b}\!\bigl(\Lambda\textup{-}\mathrm{GMod}^{+,b}\bigr)
\arrow[l, bend left=35, "\mathfrak{G}"']
\end{tikzcd}
\]
It remains to identify the categories
\(\Lambda^{!}\textup{-}\mathrm{GMod}^{-,b}\) and
\(\Lambda\textup{-}\mathrm{GMod}^{+,b}\) with the categories
\(\Lambda^{!}\textup{-Acl}^{\mathbb Z}\) and
\(\Lambda\textup{-Al}^{\mathbb Z}\), respectively.

We prove the first identification; the second is dual. Let
\(M\in\Lambda^{!}\textup{-}\mathrm{GMod}^{-,b}\). By definition,
\(M\) is right bounded and the complex \(K(M)\) has bounded cohomology.
Hence there exists an integer \(N\) such that
\[
H^r(K(M))=0
\qquad
\text{for all } r>N.
\]
By Proposition~3.1, this is equivalent to saying that the truncation
\(M_{\leq N}\) is colinear. Therefore \(M\) is almost colinear, and so
\[
\Lambda^{!}\textup{-}\mathrm{GMod}^{-,b}
=
\Lambda^{!}\textup{-Acl}^{\mathbb Z}.
\]

The dual argument gives
\[
\Lambda\textup{-}\mathrm{GMod}^{+,b}
=
\Lambda\textup{-Al}^{\mathbb Z}.
\]
Hence, the above derived Koszul duality gives
the desired equivalence
\[
\mathsf{D}^{b}\!\bigl(\Lambda^{!}\textup{-Acl}^{\mathbb Z}\bigr)
\simeq
\mathsf{D}^{b}\!\bigl(\Lambda\textup{-Al}^{\mathbb Z}\bigr).
\]
This completes the proof of the equivalence. It remains to justify the final assertion on generators. Let
\(M\in \Lambda^{!}\textup{-}\mathrm{Acl}^{\mathbb{Z}}\). By definition, there
exists an integer \(n\) such that \(M_{\leq n}\) is colinear. We have a short
exact sequence
\[
0
\longrightarrow
M_{>n}
\longrightarrow
M
\longrightarrow
M/M_{>n}
\longrightarrow
0 .
\]
Here \(M_{>n}\) is finite-dimensional, and
\(M/M_{>n}\cong M_{\leq n}\) is colinear. Under the bounded Koszul duality,
finite-dimensional graded modules are sent to bounded complexes of projective
\(\Lambda\)-modules. Hence \(M_{>n}\) belongs to the triangulated subcategory
generated by colinear \(\Lambda^{!}\)-modules. Since \(M/M_{>n}\) is colinear,
the above short exact sequence shows that \(M\) belongs to the triangulated
subcategory of
\(\mathsf{D}^{b}\!\bigl(\Lambda^{!}\textup{-}\mathrm{Acl}^{\mathbb{Z}}\bigr)\)
generated by colinear \(\Lambda^{!}\)-modules.. This
proves the final assertion.
\end{proof}

When \(\Lambda\) is finite-dimensional, the preceding theorem recovers the
graded derived Koszul duality of \cite[Theorem~3.8]{15}.

\begin{corollary}
Let \(\Lambda\) be a finite-dimensional Koszul algebra. Then
Theorem~3.2 specializes to a triangulated equivalence
\[
\begin{tikzcd}[column sep=huge]
\mathsf{D}^{b}\!\bigl(\Lambda^{!}\textup{-Cop}^{\mathbb{Z}}\bigr)
\arrow[r, bend left=35, "\mathfrak{F}"]
&
\mathsf{D}^{b}\!\bigl(\Lambda\textup{-gmod}\bigr)
\arrow[l, bend left=35, "\mathfrak{G}"']
\end{tikzcd}
\]
and, dually, to a triangulated equivalence
\[
\begin{tikzcd}[column sep=huge]
\mathsf{D}^{b}\!\bigl(\Lambda\textup{-gmod}\bigr)
\arrow[r, bend left=35, "\mathfrak{F}"]
&
\mathsf{D}^{b}\!\bigl(\Lambda^{!}\textup{-Pe}^{\mathbb{Z}}\bigr)
\arrow[l, bend left=35, "\mathfrak{G}"']
\end{tikzcd}
\]
\end{corollary}

\begin{proof}
Assume that \(\Lambda\) is finite-dimensional. By
\cite[Theorem~2.6]{15}, its Koszul dual \(\Lambda^{!}\) has finite global
dimension. In this case, the category of almost colinear
\(\Lambda^{!}\)-modules coincides with the category of coperfect
\(\Lambda^{!}\)-modules. Moreover, since \(\Lambda\) is finite-dimensional,
the category of almost linear \(\Lambda\)-modules coincides with
\(\Lambda\textup{-gmod}\), the category of finite-dimensional graded
\(\Lambda\)-modules. Therefore Theorem~3.2 gives
\[
\mathsf{D}^{b}\!\bigl(\Lambda^{!}\textup{-Cop}^{\mathbb{Z}}\bigr)
\simeq
\mathsf{D}^{b}\!\bigl(\Lambda\textup{-gmod}\bigr),
\]
as claimed.
\end{proof}
We next record a singular version of the bounded derived Koszul duality. This
generalizes the graded singular Koszul duality of
\cite[Theorem~3.16]{15}.

\begin{corollary}
The Koszul duality induces a triangulated equivalence
\[
\begin{tikzcd}[column sep=huge]
\mathsf{D}^{b}\!\bigl(
\Lambda^{!}\textup{-Acl}^{\mathbb Z}
/\Lambda^{!}\textup{-}\mathrm{gmod}
\bigr)
\arrow[r, bend left=35, "\mathfrak{F}"]
&
\mathsf{D}_{\mathrm{sg}}\!\bigl(
\Lambda\textup{-Al}^{\mathbb Z}
\bigr)
\arrow[l, bend left=35, "\mathfrak{G}"']
\end{tikzcd}
\]
\end{corollary}

\begin{proof}
The proof is identical to the proof of \cite[Theorem~3.16]{15}. Indeed, under
the graded derived Koszul duality, the bounded derived category of
finite-dimensional graded \(\Lambda^{!}\)-modules is sent to the full
subcategory of perfect complexes in
\(\mathsf{D}^{b}\!\bigl(\Lambda\textup{-}\mathrm{GMod}^{+,b}\bigr)\).
Therefore the equivalence descends to the corresponding Verdier quotients,
giving
\[
\mathsf{D}^{b}\!\bigl(
\Lambda^{!}\textup{-}\mathrm{GMod}^{-,b}
/\Lambda^{!}\textup{-}\mathrm{gmod}
\bigr)
\simeq
\mathsf{D}_{\mathrm{sg}}\!\bigl(
\Lambda\textup{-}\mathrm{GMod}^{+,b}
\bigr).
\]
\end{proof}
\subsection{Bounded Koszul duality for Quadratic Monomial Algebras}
We now specialize Theorem~3.2 to quadratic monomial algebras and show that the resulting bounded derived Koszul duality takes a particularly well-behaved form. The key point is that every finitely presented graded module over a quadratic monomial algebra is almost linear and, dually, every finitely copresented graded module is almost colinear.

\medskip

Recall that the Koszul dual \(\Lambda^{!}\) of a quadratic monomial algebra
\(\Lambda\) is again quadratic monomial. Furthermore, quadratic monomial
algebras are coherent and cocoherent; see \cite[Theorem~3.3]{17}.

\medskip 

The following lemma extends to quadratic monomial algebras the almost-linearity
properties known for commutative noetherian Koszul algebras~\cite{3} and for
noncommutative noetherian Koszul algebras with a balanced dualizing
complex~\cite{33}.

\begin{lemma}
Assume that \(\Lambda\) is a quadratic monomial algebra. Then every finitely
presented graded \(\Lambda\)-module is almost linear. Dually, every finitely
copresented graded \(\Lambda\)-module is almost colinear.
\end{lemma}

\begin{proof}
We prove only the first statement, since the second one is dual.

Let \(M\) be a finitely presented graded \(\Lambda\)-module. By
\cite[Lemme~3.2]{17}, the second syzygy of \(M\) decomposes as a finite direct
sum
\[
\Omega^{2}M \cong \bigoplus_{j=1}^{r} \Lambda \alpha_j\langle i_j\rangle
\]
for suitable arrows \(\alpha_j\) and integers \(i_j\). Moreover, by
\cite[Proposition~3.2]{17}, each module \(\Lambda \alpha_j\) is linear. Hence
\(\Omega^{2}M\) is a finite direct sum of shifts of linear modules.

Now consider the beginning of a graded projective resolution of \(M\):
\[
0 \longrightarrow \Omega^{2}M
\longrightarrow P_{1}
\longrightarrow P_{0}
\longrightarrow M
\longrightarrow 0 .
\]
Applying the coKoszul complex functor \(K\), we obtain an exact sequence of
complexes
\[
0 \longrightarrow K(\Omega^{2}M)
\longrightarrow K(P_{1})
\longrightarrow K(P_{0})
\longrightarrow K(M)
\longrightarrow 0 .
\]
Therefore we get a long exact sequence in cohomology. Since \(\Omega^{2}M\)
is a finite direct sum of shifts of linear modules, the complex
\(K(\Omega^{2}M)\) has bounded cohomology. Moreover, \(K(P_{0})\) and
\(K(P_{1})\) have bounded cohomology. It follows from the long exact sequence
that \(K(M)\) has bounded cohomology.

Thus \(M\) has a linear truncation; equivalently, \(M\) is almost linear. The
dual argument shows that every finitely copresented graded \(\Lambda\)-module
is almost colinear.
\end{proof}
We are now ready to strengthen our bounded derived Koszul duality in the case of
quadratic monomial algebras.

\begin{theorem}
Let \(\Lambda\) be a quadratic monomial algebra. Then there is an equivalence
of triangulated categories
\[
\begin{tikzcd}[column sep=huge]
\mathsf{D}^{b}\!\bigl(\Lambda^{!}\textup{-}\mathrm{Fcp}^{\mathbb{Z}}\bigr)
\arrow[r, bend left=35, "\mathfrak{F}"]
&
\mathsf{D}^{b}\!\bigl(\Lambda\textup{-}\mathrm{Fp}^{\mathbb{Z}}\bigr)
\arrow[l, bend left=35, "\mathfrak{G}"']
\end{tikzcd}
\]

\end{theorem}
In the special case where \(\Lambda\) has radical square zero, we recover the
following well-known form of Koszul duality~\cite{15}.

\begin{corollary}
Let \(\Lambda\) be a finite-dimensional algebra with radical square zero. Then
there is an equivalence of triangulated categories
\[
\begin{tikzcd}[column sep=huge]
\mathsf{D}^{b}\!\bigl(\Lambda^{!}\textup{-}\mathrm{Fcp}^{\mathbb{Z}}\bigr)
\arrow[r, bend left=35, "\mathfrak{F}"]
&
\mathsf{D}^{b}\!\bigl(\Lambda\textup{-}\mathrm{gmod}\bigr)
\arrow[l, bend left=35, "\mathfrak{G}"']
\end{tikzcd}
\]
\end{corollary}
The singular Koszul duality takes a particularly simple form for quadratic
monomial algebras.

\begin{corollary}
Let \(\Lambda\) be a quadratic monomial algebra. Then the bounded derived
Koszul duality induces an equivalence of triangulated categories
\[
\begin{tikzcd}[column sep=huge]
\mathsf{D}^{b}\!\bigl(
\operatorname{cqgr}(\Lambda^{!})
\bigr)
\arrow[r, bend left=35, "\mathfrak{F}"]
&
\mathsf{D}_{\mathrm{sg}}\!\bigl(
\Lambda\textup{-}\mathrm{Fp}^{\mathbb{Z}}
\bigr)
\arrow[l, bend left=35, "\mathfrak{G}"']
\end{tikzcd}
\]
and, dually, an equivalence
\[
\begin{tikzcd}[column sep=huge]
\mathsf{D}^{b}\!\bigl(
\Lambda^{!}\textup{-}\mathrm{Fcp}^{\mathbb{Z}}
\bigr)
\big/
\mathsf{D}^{b}\!\bigl(\Lambda^{!}\textup{-Cop}^{\mathbb{Z}}\bigr)
\arrow[r, bend left=35, "\mathfrak{F}"]
&
\mathsf{D}^{b}\!\bigl(
\operatorname{qgr}(\Lambda)
\bigr)
\arrow[l, bend left=35, "\mathfrak{G}"']
\end{tikzcd}
\]
Moreover, the tails category \(\operatorname{qgr}(\Lambda)\) and the cotails
category \(\operatorname{cqgr}(\Lambda^{!})\) are abelian hereditary
categories.
\end{corollary}

\begin{proof}
The equivalence follows from Theorem~3.6. By~\cite[Corollary~3.9]{17}, the cotails category
\[
\operatorname{cqgr}(\Lambda^{!})
\]
is abelian hereditary. This proves the assertion.
\end{proof}
If, moreover, \(\Lambda\) is Iwanaga--Gorenstein, we obtain the following
generalization of the BGG correspondence~\cite[Theorem~3.17]{15}.

\begin{corollary}
Assume that \(\Lambda\) is an Iwanaga--Gorenstein quadratic monomial algebra.
Then there are equivalences of triangulated categories
\[
\mathsf{D}^{b}\!\Bigl(
\operatorname{qgr}(\Lambda^{!})
\Bigr)
\;\xrightarrow{\ \sim\ }\;
\mathsf{D}^{b}\!\Bigl(
\operatorname{cqgr}(\Lambda^{!})
\Bigr)
\;\xrightarrow{\ \sim\ }\;
\Lambda\textup{-}\underline{\mathrm{GProj}}^{\mathbb Z} .
\]
\end{corollary}
\subsection{Bounded Koszul Duality for Certain Koszul Algebras} We next extend the bounded Koszul duality obtained for quadratic monomial
algebras in Theorem~3.6 to a broader class of Koszul algebras. This class
contains, in particular, quadratic monomial algebras and exterior algebras,
and is characterized by two homological conditions particularly well suited
to Koszul duality. Throughout this subsection, let \(\Lambda\) be a Koszul
algebra satisfying the following assumptions:
\begin{enumerate}
\item[\textup{(i)}] \(\Lambda\) is absolutely Koszul;
\item[\textup{(ii)}] every finitely presented graded \(\Lambda\)-module is
almost linear.
\end{enumerate}
We shall show that these hypotheses are sufficient to obtain a bounded
Koszul duality analogous to that established above for quadratic monomial
algebras. The proof uses results of Eisenbud, Fl{\o}ystad, and Schreyer,
originally established for the exterior algebra in
\cite[Lemma~3.5, Corollary~3.6, and Theorem~3.7]{24}. The same arguments
extend to the present setting; see also~\cite[Proposition~3.4]{56}. In the
form needed below, these results identify the linear part of a minimal
projective resolution, respectively the colinear part of a minimal
coinjective resolution, with the complex obtained by applying the Koszul
functor, respectively the coKoszul functor, to the appropriate cohomology
object over the Koszul dual algebra.

\begin{lemma}
Let \(M\in \Lambda^{!}\textup{-}\mathrm{GMod}^{-}\) and
\(N\in \Lambda\textup{-}\mathrm{GMod}^{+}\).
Let
\[
P^{\bullet}\longrightarrow N
\]
be a minimal graded projective resolution of \(N\), and let
\[
M\longrightarrow I^{\bullet}
\]
be a minimal graded injective resolution of \(M\).

Under the Koszul duality
\[
\begin{tikzcd}[column sep=huge]
\mathsf{D}^{\downarrow}\!\bigl(\Lambda^{!}\textup{-}\mathrm{GMod}^{-}\bigr)
\arrow[r, bend left=35, "\mathfrak{F}"]
&
\mathsf{D}^{\uparrow}\!\bigl(\Lambda\textup{-}\mathrm{GMod}^{+}\bigr)
\arrow[l, bend left=35, "\mathfrak{G}"']
\end{tikzcd}
\]
the colinear strand of \(I^{\bullet}\) and the linear strand of
\(P^{\bullet}\) admit the following descriptions:
\begin{enumerate}
\item[\textup{(i)}]
The colinear strand of \(I^{\bullet}\) is isomorphic to
\[
\mathfrak{G}\!\left(
\bigoplus_{n\in\mathbb Z}
H^{n}\!\bigl(\mathfrak{F}(M)\bigr)[-n]
\right).
\]

\item[\textup{(ii)}]
The linear strand of \(P^{\bullet}\) is isomorphic to
\[
\mathfrak{F}\!\left(
\bigoplus_{n\in\mathbb Z}
H^{n}\!\bigl(\mathfrak{G}(N)\bigr)[-n]
\right).
\]
\end{enumerate}
\end{lemma}

We shall also use the following simple observation.

\begin{lemma}
Let \(\Lambda\) be a Koszul algebra such that every finitely presented graded
\(\Lambda\)-module is almost linear. Then \(\Lambda\) is left graded coherent.
\end{lemma}

\begin{proof}
Let
\[
f\colon P^{-1}\longrightarrow P^{0}
\]
be a morphism between finitely generated graded projective \(\Lambda\)-modules. Then
\(\operatorname{Coker} f\) is finitely presented, and hence almost linear by assumption.
Since \(P^{0}\) is almost linear, it follows from the two-out-of-three property that
\(\operatorname{Im} f\) is almost linear. Applying the same argument to the exact sequence
\[
0\longrightarrow \operatorname{Ker} f
\longrightarrow P^{-1}
\longrightarrow \operatorname{Im} f
\longrightarrow 0,
\]
we conclude that \(\operatorname{Ker} f\) is almost linear. In particular,
\(\operatorname{Ker} f\) is finitely presented. Thus \(\Lambda\) is left graded coherent.
\end{proof}
We now use these results to establish the following dual characterization.

\begin{proposition}
Let \(\Lambda\) be a Koszul algebra with Koszul dual \(\Lambda^{!}\). Then
\(\Lambda\) satisfies conditions \textup{(i)} and \textup{(ii)} above if and only if
\(\Lambda^{!}\) satisfies the dual conditions; that is, \(\Lambda^{!}\) is absolutely
coKoszul and every finitely copresented graded \(\Lambda^{!}\)-module is almost
colinear.
\end{proposition}
\begin{proof}
Assume that \(\Lambda\) satisfies conditions \textup{(i)} and \textup{(ii)}, and let \(M\) be a finitely copresented graded \(\Lambda^{!}\)-module. Then there exists an exact sequence
\[
0\longrightarrow M\longrightarrow I^0\longrightarrow I^1.
\]

To prove that \(M\) is almost colinear, it suffices to show that the complex
\[
\mathfrak{F}(M)=K(M)
\]
has bounded cohomology. Consider the complex
\[
I^{\bullet}\colon
0\longrightarrow I^0\longrightarrow I^1\longrightarrow 0,
\]
such that
\[
H^0(I^{\bullet})=M.
\]

Since \(I^{\bullet}\) is bounded, the complex \(\mathfrak{F}(I^{\bullet})\) has bounded cohomology. Let
\[
n=\min\{\,i\in\mathbb Z \mid H^{i}(\mathfrak{F}(I^{\bullet}))\neq 0\,\}.
\]
Then there is an exact triangle in
\[
\mathsf{D}^{\uparrow}\!\bigl(\Lambda\textup{-}\mathrm{GMod}^{+}\bigr)
\]
of the form
\[
\mathfrak{F}(I^{\bullet})^{\geq n}
\longrightarrow
\mathfrak{F}(I^{\bullet})
\longrightarrow
\mathfrak{F}(I^{\bullet})^{< n}
\longrightarrow
\mathfrak{F}(I^{\bullet})^{\geq n}[1].
\]

Since \(n\) is the smallest integer for which
\(H^{n}(\mathfrak{F}(I^{\bullet}))\neq 0\), the complex
\(\mathfrak{F}(I^{\bullet})^{< n}\) is a projective resolution of
\(\operatorname{Im} d^{\,n-1}\). Hence we obtain an exact triangle
\[
\mathfrak{F}(I^{\bullet})^{\geq n}
\longrightarrow
\mathfrak{F}(I^{\bullet})
\longrightarrow
\operatorname{Im} d^{\,n-1}[-n]
\longrightarrow
\mathfrak{F}(I^{\bullet})^{\geq n}[1].
\]

Applying \(\mathfrak{G}\) yields an exact triangle in
\[
\mathsf{D}^{\downarrow}\!\bigl(\Lambda^{!}\textup{-}\mathrm{GMod}^{-}\bigr)
\]
\[
\mathfrak{G}\!\left(\mathfrak{F}(I^{\bullet})^{\geq n}\right)
\longrightarrow
\mathfrak{G}\!\left(\mathfrak{F}(I^{\bullet})\right)
\longrightarrow
\mathfrak{G}\!\left(\operatorname{Im} d^{\,n-1}[-n]\right)
\longrightarrow
\mathfrak{G}\!\left(\mathfrak{F}(I^{\bullet})^{\geq n}\right)[1].
\]

Since \(\mathfrak{G}\mathfrak{F}\cong \operatorname{Id}\), this triangle may be rewritten as
\[
\mathfrak{G}\!\left(\mathfrak{F}(I^{\bullet})^{\geq n}\right)
\longrightarrow
I^{\bullet}
\longrightarrow
\mathfrak{G}\!\left(\operatorname{Im} d^{\,n-1}[-n]\right)
\longrightarrow
\mathfrak{G}\!\left(\mathfrak{F}(I^{\bullet})^{\geq n}\right)[1].
\]

Since \(\mathfrak{F}(I^{\bullet})^{\geq n}\) is a bounded complex of projective modules, the complex
\[
\mathfrak{G}\!\left(\mathfrak{F}(I^{\bullet})^{\geq n}\right)
\]
has bounded cohomology, and all of its nonzero cohomology modules are finite-dimensional. On the other hand, since \(\Lambda\) is absolutely Koszul, Lemma~3.10\textup{(ii)} implies that the cohomology modules of
\[
\mathfrak{G}\!\left(\operatorname{Im} d^{\,n-1}[-n]\right)
\]
are almost colinear.

Therefore, the long exact sequence in cohomology, together with the two-out-of-three property of the category
\[
\Lambda^{!}\textup{-}\mathrm{GMod}^{-,b},
\]
shows that
\[
H^0(I^{\bullet})=M
\]
belongs to \(\Lambda^{!}\textup{-}\mathrm{GMod}^{-,b}\). Consequently,
\(\mathfrak{F}(M)\) has bounded cohomology, and hence \(M\) is almost colinear.

It remains to prove that \(\Lambda^{!}\) is absolutely coKoszul. To this end, let \(M\) be a finitely copresented graded \(\Lambda^{!}\)-module. Since \(\mathfrak{F}(M)\) has bounded cohomology and every finitely presented graded \(\Lambda\)-module is almost linear, it follows that each cohomology module
\[
H^{n}\!\bigl(\mathfrak{F}(M)\bigr)
\]
is almost linear. Equivalently,
\[
\mathfrak{G}\!\left(H^{n}\!\bigl(\mathfrak{F}(M)\bigr)\right)
\]
has bounded cohomology for every \(n\in\mathbb Z\).

Hence, by Lemma~3.10\textup{(i)}, the colinear strand of a minimal graded injective resolution of \(M\) has bounded cohomology. Therefore \(M\) admits a weakly coKoszul cosyzygy, and consequently \(\Lambda^{!}\) is absolutely coKoszul.

The converse implication is proved dually.
\end{proof}
We are now ready to state the main result of this subsection.

\begin{theorem}
Let \(\Lambda\) be a Koszul algebra satisfying conditions \textup{(i)} and \textup{(ii)}. Then there is a Koszul duality
\[
\begin{tikzcd}[column sep=huge]
\mathsf{D}^{b}\!\bigl(\Lambda^{!}\textup{-}\mathrm{Fcp}^{\mathbb{Z}}\bigr)
\arrow[r, bend left=35, "\mathfrak{F}"]
&
\mathsf{D}^{b}\!\bigl(\Lambda\textup{-}\mathrm{Fp}^{\mathbb{Z}}\bigr)
\arrow[l, bend left=35, "\mathfrak{G}"']
\end{tikzcd}
\]
\end{theorem}
\begin{proof}
By Lemma~3.11, the algebra \(\Lambda\) is left graded coherent and its Koszul dual
\(\Lambda^{!}\) is left graded cocoherent. Moreover, the category of almost linear
graded \(\Lambda\)-modules coincides with the category of finitely presented graded
\(\Lambda\)-modules, while the category of almost colinear graded \(\Lambda^{!}\)-modules
coincides with the category of finitely copresented graded \(\Lambda^{!}\)-modules.
The asserted equivalence therefore follows from the bounded graded Koszul duality
for almost linear and almost colinear modules established in Theorem~3.2.
\end{proof}
We also obtain a singular Koszul duality for this class of Koszul algebras.

\begin{corollary}
Let \(\Lambda\) be a Koszul algebra satisfying conditions \textup{(i)} and \textup{(ii)}. Then there is a singular Koszul duality
\[
\begin{tikzcd}[column sep=huge]
\mathsf{D}^{b}\!\bigl(
\operatorname{cqgr}(\Lambda^{!})
\bigr)
\arrow[r, bend left=35, "\mathfrak{F}"]
&
\mathsf{D}_{\mathrm{sg}}\!\bigl(
\Lambda\textup{-}\mathrm{Fp}^{\mathbb{Z}}
\bigr)
\arrow[l, bend left=35, "\mathfrak{G}"']
\end{tikzcd}
\]
together with its dual counterpart
\[
\begin{tikzcd}[column sep=huge]
\mathsf{D}^{b}\!\bigl(
\Lambda^{!}\textup{-}\mathrm{Fcp}^{\mathbb{Z}}
\bigr)
\big/
\mathsf{D}^{b}\!\bigl(\Lambda^{!}\textup{-Cop}^{\mathbb{Z}}\bigr)
\arrow[r, bend left=35, "\mathfrak{F}"]
&
\mathsf{D}^{b}\!\bigl(
\operatorname{qgr}(\Lambda)
\bigr)
\arrow[l, bend left=35, "\mathfrak{G}"']
\end{tikzcd}
\]
\end{corollary}
\section{Applications to Geometry}
The purpose of this final section is to explain how the preceding results
yield a BGG-type description for Koszul quotients of polynomial algebras.
More precisely, let \(\Lambda\) be a Koszul quotient of a polynomial
algebra and let \(X=\operatorname{Proj}(\Lambda)\). We obtain a
Koszul-dual description of the bounded derived category
\(\mathsf{D}^{b}\!\bigl(\operatorname{Coh}(X)\bigr)\). In this sense,
our results extend the classical Bernstein--Gel'fand--Gel'fand
correspondence from projective space to projective schemes defined by
Koszul quotients of polynomial algebras.

More generally, if \(A\) is a standard graded quotient of a polynomial
algebra, then sufficiently high Veronese subalgebras \(A^{(d)}\) are
Koszul. Since
\(\operatorname{Proj}(A)\cong \operatorname{Proj}(A^{(d)})\),
the preceding construction can therefore be applied, after passing to a
suitable Veronese subalgebra, to arbitrary closed projective subschemes
of projective space. Thus, the bounded derived category of coherent
sheaves on such a projective scheme admits a description in terms of the
Koszul dual of an appropriate Veronese algebra.

Since the Koszul dual of a Koszul quotient of a polynomial algebra is
naturally presented as a quotient of a free associative algebra by
quadratic relations, this description realizes a geometric category of
coherent sheaves as a Verdier quotient associated with a generally
noncommutative algebra. In this way, the preceding results provide a
natural bridge between algebraic geometry, Koszul duality, and
noncommutative projective geometry.

We also revisit Beilinson's description of the bounded derived category
of coherent sheaves on projective space and extend this point of view to
bounded derived categories of quotient categories associated with
generalized Artin--Schelter regular algebras in the sense of
Mart\'inez-Villa.
\subsection{Koszulity in Algebraic Geometry} As recalled in the introduction, Serre's theorem~\cite{52} identifies the category of coherent sheaves on the projective scheme \(X=\operatorname{Proj}(\Lambda)\), where \(\Lambda\) is a commutative algebra generated in degree \(1\), with the quotient category \[ \operatorname{coh}(X) \cong \operatorname{qgr}(\Lambda). \] A noncommutative analogue of this construction was introduced by Artin and Zhang~\cite{2} in the framework of noncommutative projective geometry. The result below concerns the case where \(\Lambda\) is a Koszul quotient of a polynomial algebra. 

\medskip

We begin with an important result of Avramov and Eisenbud~\cite{3}. \begin{lemma} Let \(\Lambda\) be a commutative noetherian Koszul algebra. Then every finitely generated graded \(\Lambda\)-module is almost linear. \end{lemma}
We now prove the main result of this subsection, which gives a Koszul-dual
description of the bounded derived category
\(\mathsf{D}^{b}\!\bigl(\operatorname{Coh}(X)\bigr)\). 

\begin{theorem} Let \(\Lambda\) be a Koszul quotient of a polynomial algebra, and set \[ X=\operatorname{Proj}(\Lambda). \] Then \(\Lambda^{!}\) is absolutely coKoszul, and there are triangulated equivalences \[ \mathsf{D}^{b}\!\bigl(\operatorname{coh}(X)\bigr) \cong \mathsf{D}^{b}\!\bigl(\Lambda^{!}\textup{-}\mathrm{Acl}^{\mathbb{Z}}\bigr) \big/ \mathsf{D}^{b}\!\bigl(\Lambda^{!}\textup{-}\mathrm{Cop}^{\mathbb{Z}}\bigr) \cong \mathsf{S}\!\bigl(\Lambda^{!}\textup{-}\overline{\mathrm{Acl}^{\mathbb{Z}}}\bigr). \] Under the first equivalence, the category \(\operatorname{coh}(X)\) identifies with the category consisting of cohomological shifts of colinear \(\Lambda^{!}\)-modules of infinite injective dimension. Moreover, every object of the Verdier quotient \[ \mathsf{D}^{b}\!\bigl(\Lambda^{!}\textup{-}\mathrm{Acl}^{\mathbb{Z}}\bigr) \big/ \mathsf{D}^{b}\!\bigl(\Lambda^{!}\textup{-}\mathrm{Cop}^{\mathbb{Z}}\bigr) \] is isomorphic to a cohomological shift of a weakly coKoszul \(\Lambda^{!}\)-module. \end{theorem}\begin{proof} The absolute co-Koszulity of \(\Lambda^{!}\) follows from Lemma~4.1 and Proposition~3.12. By Lemma~4.1, the category \(  \Lambda\textup{-}\mathrm{Al}^{\mathbb{Z}} \) coincides with \( \Lambda\textup{-}\mathrm{Fg}^{\mathbb{Z}}, \) the category of finitely generated graded \(\Lambda\)-modules. Hence the graded derived Koszul duality of Theorem~3.2 specializes to a triangulated equivalence \[ \begin{tikzcd}[column sep=huge] \mathsf{D}^{b}\!\bigl(\Lambda^{!}\textup{-}\mathrm{Acl}^{\mathbb{Z}}\bigr) \arrow[r, bend left=35, "\mathfrak{F}"] & \mathsf{D}^{b}\!\bigl(\Lambda\textup{-}\mathrm{Fg}^{\mathbb{Z}}\bigr) \arrow[l, bend left=35, "\mathfrak{G}"'] \end{tikzcd} \] Under this equivalence, the category of coperfect complexes on the left-hand side corresponds to the bounded derived category of finite-dimensional graded \(\Lambda\)-modules. Therefore the equivalence descends to the corresponding Verdier quotients and gives \[ \mathsf{D}^{b}\!\bigl(\Lambda^{!}\textup{-}\mathrm{Acl}^{\mathbb{Z}}\bigr) \big/ \mathsf{D}^{b}\!\bigl(\Lambda^{!}\textup{-}\mathrm{Cop}^{\mathbb{Z}}\bigr) \cong \mathsf{D}^{b}\!\bigl(\Lambda\textup{-}\mathrm{Fg}^{\mathbb{Z}}\bigr) \big/ \mathsf{D}^{b}\!\bigl(\Lambda\textup{-}\mathrm{gmod}\bigr). \] Since \( \operatorname{qgr}(\Lambda) = \Lambda\textup{-}\mathrm{Fg}^{\mathbb{Z}} \big/ \Lambda\textup{-}\mathrm{gmod}, \) the quotient on the right-hand side identifies with \(  \mathsf{D}^{b}\!\bigl(\operatorname{qgr}(\Lambda)\bigr). \) Combining these identifications, we obtain the desired triangulated equivalence \[ \mathsf{D}^{b}\!\bigl(\operatorname{coh}(X)\bigr) \cong \mathsf{D}^{b}\!\bigl(\Lambda^{!}\textup{-}\mathrm{Acl}^{\mathbb{Z}}\bigr) \big/ \mathsf{D}^{b}\!\bigl(\Lambda^{!}\textup{-}\mathrm{Cop}^{\mathbb{Z}}\bigr). \] It remains to identify the objects corresponding to coherent sheaves. By assumption, every object of \(\operatorname{qgr}(\Lambda)\) is isomorphic to a linear \(\Lambda\)-module. By Proposition~3.1, the functor \(\mathfrak{G}\) sends linear \(\Lambda\)-modules to cohomological shifts of colinear \(\Lambda^{!}\)-modules. Therefore, under the above equivalence, the category \(\operatorname{coh}(X)\) identifies with the category consisting of cohomological shifts of colinear \(\Lambda^{!}\)-modules. Finally, since \(\Lambda^{!}\) is absolutely co-Koszul, every almost-colinear \(\Lambda^{!}\)-module has a weakly co-Koszul cosyzygy. Hence, after passing to the Verdier quotient by \(  \mathsf{D}^{b}\!\bigl(\Lambda^{!}\textup{-}\mathrm{Cop}^{\mathbb{Z}}\bigr), \) every object is isomorphic to a cohomological shift of a weakly co-Koszul \(\Lambda^{!}\)-module. This proves the final assertion. \end{proof}
 Theorem 4.2 yields the following general
description.

\begin{corollary}
Let \(X\subseteq \mathbb{P}_k^n\) be a closed projective subscheme, and
let \(A\) be its homogeneous coordinate ring. Choose \(d\) sufficiently
large so that the Veronese subalgebra \(A^{(d)}\) is Koszul, and set
\(\Lambda=A^{(d)}\). Then there are triangulated equivalences
\[
\mathsf{D}^{b}\!\bigl(\operatorname{Coh}(X)\bigr)
\cong
\mathsf{D}^{b}\!\bigl(
\Lambda^{!}\textup{-}\mathrm{Acl}^{\mathbb{Z}}
\bigr)
\big/
\mathsf{D}^{b}\!\bigl(
\Lambda^{!}\textup{-}\mathrm{Cop}^{\mathbb{Z}}
\bigr)
\cong
\mathsf{S}\!\bigl(
\Lambda^{!}\textup{-}\overline{\mathrm{Acl}^{\mathbb{Z}}}
\bigr).
\]
Here \(\Lambda^{!}=(A^{(d)})^{!}\) denotes the Koszul dual of the
Veronese Koszul algebra \(A^{(d)}\) associated with the homogeneous
coordinate ring of \(X\).
\end{corollary}
\begin{example}

When
\( 
X=\mathbb{P}^{n}
=
\operatorname{Proj}\!\bigl(k[x_{0},\dots,x_{n}]\bigr),
\)
the preceding theorem recovers the classical
Bernstein--Gelfand--Gelfand correspondence. In this case, the Koszul dual is
the exterior algebra
\[
E=\bigwedge V,
\qquad
V=k^{n+1}.
\]
In particular, \(E\) is finite-dimensional and self-injective. Hence
Theorem~4.2 specializes to the following statement.
There are triangulated equivalences
\[
\mathsf{D}^{b}\!\bigl(\operatorname{Coh}(\mathbb{P}^{n})\bigr)
\cong
\mathsf{D}^{b}\!\bigl(E\textup{-}\mathrm{gmod}\bigr)
\big/
\mathsf{D}^{b}\!\bigl(E\textup{-Cop}^{\mathbb{Z}}\bigr)
\cong \mathsf{S}\!\bigl(E\textup{-}\overline{\mathrm{gmod}}\bigr)\cong 
E\textup{-}\underline{\mathrm{gmod}}.
\]
\end{example}
Theorem 4.2 becomes particularly concrete for Koszul complete intersections, which form a natural class of absolutely Koszul algebras whose Koszul duals enjoy strong finiteness properties. The following example illustrates this situation in the geometric setting.

\begin{example}
Let \(\Lambda=k[x_1,\ldots,x_n]/(q_1,\ldots,q_c)\), where \(q_1,\ldots,q_c\) is a regular sequence of quadratic forms. Then \(\Lambda\) is  Koszul and, in particular, is absolutely Koszul; see \cite{31}. Moreover, its Koszul dual \(\Lambda^{!}\) is left noetherian; see \cite{11}.

By Proposition~3.12, \(\Lambda^{!}\) is absolutely coKoszul and every finitely copresented graded module is almost colinear. So the category of almost colinear modules coincides with the category of finitely copresented graded modules. Since in our setting \(\Lambda^{!}\) is left conoetherian, the latter category coincides with \(\Lambda^{!}\textup{-}\mathrm{Fcg}^{\mathbb{Z}}\).

If \(X=\operatorname{Proj}(\Lambda)\), then \(X\subseteq\mathbb{P}^{n-1}\) is a projective complete intersection of quadrics. Therefore, Theorem~4.2 gives triangulated equivalences
\[
\mathsf{D}^{b}\!\bigl(\operatorname{coh}(X)\bigr)
\cong
\mathsf{D}^{b}\!\bigl(\Lambda^{!}\textup{-}\mathrm{Fcg}^{\mathbb{Z}}\bigr)
\big/
\mathsf{D}^{b}\!\bigl(\Lambda^{!}\textup{-}\mathrm{Cop}^{\mathbb{Z}}\bigr)
\cong
\mathsf{S}\!\bigl(\Lambda^{!}\textup{-}\overline{\mathrm{Fcg}^{\mathbb{Z}}}\bigr).
\]
\end{example}
\subsection{Koszulity in Noncommutative Geometry} 

Beilinson~\cite{6} proved that the bounded derived category of coherent
sheaves on projective space admits a full strong exceptional sequence. As a
consequence, one obtains a triangulated equivalence
\[
\mathsf{D}^{b}\!\bigl(\operatorname{Coh}(\mathbb{P}^{d})\bigr)
\simeq
\mathsf{D}^{b}\!\bigl(B_{d}\textup{-}\mathrm{mod}\bigr),
\]
where \(B_{d}\) is the corresponding Beilinson algebra. In particular,
\(B_{d}\) is finite-dimensional and has finite global dimension. This
equivalence is commonly referred to as Beilinson's theorem.

The aim of this section is to establish an analogue of Beilinson's theorem
in a noncommutative setting, using the quotient categories introduced by
Artin and Zhang. More precisely, for a finite-dimensional self-injective
Koszul algebra \(\Lambda\), we study the bounded derived category
\( 
\mathsf{D}^{b}\!\bigl(\operatorname{qgr}(\Lambda^{!})\bigr)
\)
and show that it admits a full strong exceptional sequence whose direct sum
is a tilting object. We further identify the corresponding endomorphism
algebra as a finite-dimensional Koszul algebra of finite global dimension.

We first recall the notion of generalized Artin--Schelter regularity that
will be used below. Let
\(A=kQ/I\), where \(Q\) is finite and \(I\) is homogeneous. We say that
\(A\) is \emph{generalized Artin--Schelter regular}, or simply
\emph{generalized AS-regular}, of dimension \(d\) if the following
conditions hold:
\begin{enumerate}
\item \(\operatorname{gldim}A=d<\infty\).

\item For every graded simple \(A\)-module \(S\),
\( 
\operatorname{Ext}^{i}_{A}(S,A)=0
\qquad\text{for all }i\neq d.
\)

\item The functor
\( 
\operatorname{Ext}^{d}_{A}(-,A)
\)
induces a bijection between the isomorphism classes of graded simple
\(A\)-modules and those of graded simple \(A^{\mathrm{op}}\)-modules.
\end{enumerate}

This notion extends the classical notion of
Artin--Schelter regularity and is particularly well suited to nonconnected
graded algebras arising from bound quivers; see~\cite{40,42}.
Classical Artin--Schelter regular algebras were introduced as
noncommutative analogues of polynomial algebras and form a fundamental
class in noncommutative algebraic geometry; see~\cite{1,2}. Generalized
AS-regular algebras provide a broader framework in which analogous
homological phenomena persist.

This framework is particularly natural from the viewpoint of Koszul
duality. Indeed, for a Koszul algebra of finite global dimension,
generalized AS-regularity is equivalent to self-injectivity of its Koszul
dual; see~\cite{39,40}. Consequently, if \(\Lambda\) is a finite-dimensional
self-injective Koszul algebra, then its Koszul dual \(\Lambda^{!}\) is a
generalized AS-regular Koszul algebra. Thus, the class of algebras
considered here arises naturally, through Koszul duality, from
finite-dimensional self-injective Koszul algebras.

Since \(\operatorname{qgr}(\Lambda^{!})\) plays the role of the category of
coherent sheaves on a noncommutative projective space, the resulting
description of
\( 
\mathsf{D}^{b}\!\bigl(\operatorname{qgr}(\Lambda^{!})\bigr)
\)
may be viewed as a noncommutative analogue of Beilinson's description of
\( 
\mathsf{D}^{b}\!\bigl(\operatorname{Coh}(\mathbb{P}^{d})\bigr).
\)

We next recall the notions of exceptional sequences, full strong
exceptional sequences, and tilting objects. Let \(\mathcal{T}\) be a
\(k\)-linear triangulated category. An object \(E\in\mathcal{T}\) is called
\emph{exceptional} if
\( 
\operatorname{Hom}_{\mathcal{T}}(E,E)\cong k
\)
and
\( 
\operatorname{Hom}_{\mathcal{T}}(E,E[i])=0\)
\(\text{for all }i\neq0.
\)

A sequence
\( 
(E_{0},E_{1},\ldots,E_{m})
\)
of exceptional objects is called an \emph{exceptional sequence} if
\( 
\operatorname{Hom}_{\mathcal{T}}(E_{j},E_{i}[r])=0
\)
for all \(j>i\) and all \(r\in\mathbb{Z}\). It is called \emph{strong} if,
in addition,
\( 
\operatorname{Hom}_{\mathcal{T}}(E_{i},E_{j}[r])=0
\)
for all \(i,j\) and all \(r\neq0\), and it is called \emph{full} if
\( 
\operatorname{thick}(E_{0},E_{1},\ldots,E_{m})=\mathcal{T}.
\)

An object \(T\in\mathcal{T}\) is called a \emph{tilting object} if
\( 
\operatorname{Hom}_{\mathcal{T}}(T,T[i])=0,\) 
\(\text{for all }i\neq0
\)
and
\( 
\operatorname{thick}(T)=\mathcal{T}.
\)
Thus, if
\( 
(E_{0},E_{1},\ldots,E_{m})
\)
is a full strong exceptional sequence, then
\( 
T=\bigoplus_{i=0}^{m}E_{i}
\)
is a tilting object.

Bondal~\cite{12} showed that if 
\(\mathcal{T}\) is generated by a full strong exceptional sequence
\((E_{0},E_{1},\ldots,E_{m})\), then
\[
\mathcal{T}
\simeq
\mathsf{D}^{b}\!\bigl(B\textup{-}\mathrm{mod}\bigr),
\]
where \(T=\bigoplus_{i=0}^{m}E_i\),
\(B=\operatorname{End}_{\mathcal{T}}(T)\) is a finite-dimensional
algebra of finite global dimension, and the equivalence is induced by the functor
\( 
\mathbf{R}\!\operatorname{Hom}_{\mathcal{T}}(T,-).
\)

Bondal also investigated conditions under which the algebra \(B\) is Koszul.
Such endomorphism algebras play an important role in the study of derived
categories of coherent sheaves.

Following~\cite{12}, we call a finite quiver \(Q\) \emph{ordered} if its
vertices can be arranged in such a way that every arrow is oriented from
a smaller vertex to a larger one. In particular, an ordered quiver is
acyclic, and every finite-dimensional bound path algebra associated with
such a quiver has finite global dimension. In~\cite{15}, we called such a
quiver \emph{well directed}.

In order to establish a relationship between the quivers of \(\Lambda^{!}\)
and \(B\), we recall some well-known facts about graded algebras. To any
graded algebra \(A\), one can associate a \(\mathbb Z\)-category
\(A^{\mathbb Z}\), and the category \(A\textup{-}\mathrm{GMod}\) is
equivalent to the category \(A^{\mathbb Z}\textup{-}\mathrm{Mod}\) of
locally finite-dimensional \(A^{\mathbb Z}\)-modules.

To illustrate this construction, let \(A\) be the graded algebra associated
with the quiver
\[
\begin{tikzcd}[row sep=large, column sep=large]
x \arrow[r,"\alpha"]
& y \arrow[loop right, distance=8mm, "\beta"]
\end{tikzcd}
\]
subject to the relation
\[
\beta^2=0.
\]
Then the quiver \(Q^{\mathbb Z}\) of \(A^{\mathbb Z}\) is given by
\[
\begin{tikzcd}
\cdots & (x,-1) \arrow[dr,"{(\alpha,-1)}"]
& (x,0) \arrow[dr,"{(\alpha,0)}"]
& (x,1) \arrow[dr,"{(\alpha,1)}"]
& (x,2)
& \cdots \\
\cdots & (y,-1) \arrow[r,"{(\beta,-1)}"]
& (y,0) \arrow[r,"{(\beta,0)}"]
& (y,1) \arrow[r,"{(\beta,1)}"]
& (y,2)
& \cdots
\end{tikzcd}
\]
with relations
\[
(\beta,n)(\beta,n-1)=0
\qquad\text{for all }n\in\mathbb Z.
\]

Note that \(A^{\mathbb Z}\) is a directed \(k\)-linear category in the sense
of~\cite{29}.

Let \(A=kQ/I\) be a graded algebra, and let \(A^{\mathbb Z}\) denote its
associated \(\mathbb Z\)-category. A full subcategory
\(\mathcal{C}\subseteq A^{\mathbb Z}\) is called \emph{convex} if, whenever
\(x,y\in\operatorname{Ob}(\mathcal{C})\) and
\[
x\leq z\leq y
\]
for some \(z\in\operatorname{Ob}(A^{\mathbb Z})\), then
\(z\in\operatorname{Ob}(\mathcal{C})\).

Equivalently, if a path in the quiver of \(A^{\mathbb Z}\) starts and ends
at vertices of \(\mathcal{C}\), then every intermediate vertex of that path
also belongs to \(\mathcal{C}\).

We call the algebra associated with a finite full convex subcategory of
\(A^{\mathbb Z}\) a \emph{convex truncation algebra} of \(A^{\mathbb Z}\).
By~\cite[Theorem~3.5]{29}, \(A^{\mathbb Z}\) is Koszul if and only if every
full convex subcategory of \(A^{\mathbb Z}\) is Koszul. In particular, every
convex truncation algebra of a Koszul \(\mathbb Z\)-category
\(A^{\mathbb Z}\) is Koszul.

We are now in a position to prove the main result of this subsection.

\begin{theorem}
Let \(\Lambda\) be a finite-dimensional self-injective Koszul algebra, so that its Koszul dual \(\Lambda^{!}\) is generalized AS-regular. Then there exists a finite-dimensional Koszul algebra \(B\) of finite global dimension, whose quiver is ordered, together with triangulated equivalences
\[
\mathsf{D}^{b}\!\bigl(\operatorname{qgr}(\Lambda^{!})\bigr)
\simeq
\mathsf{D}^{b}\!\bigl(B\textup{-}\mathrm{mod}\bigr)
\simeq
\mathsf{D}^{b}\!\bigl(B^{!}\textup{-}\mathrm{mod}\bigr).
\]
\end{theorem}
\begin{proof}
By Corollary~3.4, after interchanging the two triangulated categories
\(\mathsf{D}_{\mathrm{sg}}(\Lambda\textup{-}\mathrm{gmod})\) and
\(\mathsf{D}^{b}(\operatorname{qgr}(\Lambda^{!}))\), the singular Koszul
duality yields a triangulated equivalence
\[
\mathfrak{F}\colon
\mathsf{D}_{\mathrm{sg}}\!\bigl(\Lambda\textup{-}\mathrm{gmod}\bigr)
\xrightarrow{\sim}
\mathsf{D}^{b}\!\bigl(\operatorname{qgr}(\Lambda^{!})\bigr).
\]

Now, since \(\Lambda\) is self-injective, every \(\Lambda\)-module is
Gorenstein-projective, and in particular every simple \(\Lambda\)-module is
Gorenstein-projective. Hence, by Corollary~6.4.1(i) in~\cite{20}, for all
\(x,y\in Q_0\) and all \(i,j,n\in\mathbb Z\), there are natural isomorphisms
of \(k\)-vector spaces
\[
\operatorname{Hom}_{\mathsf{D}_{\mathrm{sg}}
(\Lambda\textup{-}\mathrm{gmod})}
\bigl(
S_x\langle i\rangle,
S_y\langle j\rangle[n]
\bigr)
\cong
\operatorname{Hom}_{\mathsf{D}^{b}
(\Lambda\textup{-}\mathrm{gmod})}
\bigl(
S_x\langle i\rangle,
S_y\langle j\rangle[n]
\bigr).
\]

Moreover, since \(\Lambda\) is Koszul,
\[
\operatorname{Hom}_{\mathsf{D}_{\mathrm{sg}}
(\Lambda\textup{-}\mathrm{gmod})}
\bigl(
S_x\langle i\rangle,
S_y\langle j\rangle[n]
\bigr)
=0
\]
for all \(n\neq i-j\).

Since
\[
\mathfrak{F}\bigl(S_x\langle i\rangle\bigr)
\cong
P_x^{!}\langle-i\rangle[i],
\]
it follows that
\[
\operatorname{Hom}_{\mathsf{D}^{b}
(\operatorname{qgr}(\Lambda^{!}))}
\bigl(
P_x^{!}\langle-i\rangle[i],
P_y^{!}\langle-j\rangle[j+n]
\bigr)
=0
\]
for all \(n\neq i-j\). Equivalently,
\[
\operatorname{Hom}_{\mathsf{D}^{b}
(\operatorname{qgr}(\Lambda^{!}))}
\bigl(
P_x^{!}\langle-i\rangle,
P_y^{!}\langle-j\rangle[j-i+n]
\bigr)
=0
\]
for all \(n\neq i-j\).

Invoking the derived Koszul duality
\[
\mathsf{D}^{b}\!\bigl(\Lambda\textup{-}\mathrm{gmod}\bigr)
\simeq
\mathsf{D}^{b}\!\bigl(\Lambda^{!}\textup{-}\mathrm{Pe}^{\mathbb Z}\bigr),
\]
we obtain
\[
\begin{aligned}
\operatorname{Hom}_{\mathsf{D}^{b}
(\operatorname{qgr}(\Lambda^{!}))}
\bigl(
P_x^{!}\langle-i\rangle,
P_y^{!}\langle-j\rangle
\bigr)
&\cong
\operatorname{Hom}_{\mathsf{D}^{b}
(\Lambda^{!}\textup{-}\mathrm{Pe}^{\mathbb Z})}
\bigl(
P_x^{!}\langle-i\rangle,
P_y^{!}\langle-j\rangle
\bigr)
\\
&\cong
\operatorname{Hom}_{\Lambda^{!}\textup{-}\mathrm{GMod}^{+}}
\bigl(
P_x^{!}\langle-i\rangle,
P_y^{!}\langle-j\rangle
\bigr).
\end{aligned}
\]

Now, since \(\Lambda^{!}\) has finite global dimension, say \(d\), each
simple graded \(\Lambda^{!}\)-module \(S_x^{!}\) generated in degree zero
admits a finite linear projective resolution
\[
0
\longrightarrow
P^{-d}
\longrightarrow
P^{-d+1}
\longrightarrow
\cdots
\longrightarrow
P^{-2}
\longrightarrow
P^{-1}
\longrightarrow
P^{0}
\longrightarrow
S_x^{!}
\longrightarrow
0.
\]
Since the quotient functor
\[
\pi\colon
\Lambda^{!}\textup{-}\mathrm{Pe}^{\mathbb Z}
\longrightarrow
\operatorname{qgr}(\Lambda^{!})
\]
is exact and \(S_x^{!}\) vanishes in
\(\operatorname{qgr}(\Lambda^{!})\), it sends the above resolution to an
exact complex
\[
0
\longrightarrow
\pi P^{-d}
\longrightarrow
\pi P^{-d+1}
\longrightarrow
\cdots
\longrightarrow
\pi P^{-2}
\longrightarrow
\pi P^{-1}
\longrightarrow
\pi P^{0}
\longrightarrow
0
\]
in \(\operatorname{qgr}(\Lambda^{!})\).

Since the resolution is linear, each \(P^{-k}\) is a finite direct sum of
modules of the form \(P_y^{!}\langle-k\rangle\), with \(y\in Q_0\).
Consequently, each \(\pi P^{-k}\) is a finite direct sum of objects of the
form
\( 
\pi P_y^{!}\langle-k\rangle.
\)

Applying the same argument to all grading shifts of the simple modules
shows recursively that every object of the form
\(\pi P_x^{!}\langle-r\rangle\), with \(r\in\mathbb Z\), belongs to the
triangulated subcategory generated by
\[
\left\{
\pi P_x^{!}\langle-i\rangle
\;\middle|\;
x\in Q_0,\ 0\leq i\leq d-1
\right\}.
\]
Hence the finite family
\[
\left\{
\pi P_x^{!}\langle-i\rangle
\;\middle|\;
x\in Q_0,\ 0\leq i\leq d-1
\right\},
\]
or, equivalently, the family
\[
\left\{
\pi\Lambda^{!}\langle-i\rangle
\;\middle|\;
0\leq i\leq d-1
\right\},
\]
generates
\(\mathsf{D}^{b}(\operatorname{qgr}(\Lambda^{!}))\)
as a triangulated category.

Moreover, the preceding Hom-vanishing, together with the degree-zero
Hom-identification above, shows that, after choosing a compatible ordering,
the objects
\[
\left\{
\pi P_x^{!}\langle-i\rangle
\;\middle|\;
x\in Q_0,\ 0\leq i\leq d-1
\right\}
\]
form a full strong exceptional collection in
\(\mathsf{D}^{b}(\operatorname{qgr}(\Lambda^{!}))\).

Set
\[
B=
\operatorname{End}_{\operatorname{qgr}(\Lambda^{!})}
\left(
\bigoplus_{i=0}^{d-1}\pi\Lambda^{!}\langle-i\rangle
\right)
\cong
\operatorname{End}_{\Lambda^{!}\textup{-}\mathrm{GMod}^{+}}
\left(
\bigoplus_{i=0}^{d-1}\Lambda^{!}\langle-i\rangle
\right).
\]
Then \(B\) is finite-dimensional, and its quiver is ordered. Hence \(B\)
has finite global dimension. 

It remains to prove that \(B\) is Koszul. By construction, \(B\) is a
convex truncation algebra of \((\Lambda^{!})^{\mathbb Z}\) determined by
\(d\) consecutive levels. Therefore, \(B\) is Koszul by
\cite[Theorem~3.5]{29}. The second equivalence follows from
\cite[Corollary~3.12]{15}.
\end{proof}
We now derive two consequences of Theorem~4.6. The first recovers Beilinson's theorem~\cite{6} and provides a new description of
\( 
\mathsf{D}^{b}\bigl(\operatorname{Coh}(\mathbb{P}^{d})\bigr)
\)
in terms of the Koszul dual of the Beilinson algebra \(B\), while the second recovers the noncommutative generalization due to Minamoto~\cite{45}.

Beilinson showed that
\[
\bigl(\mathcal{O},\mathcal{O}\langle-1\rangle,\ldots,
\mathcal{O}\langle-d\rangle\bigr)
\]
is a full strong exceptional sequence in
\(\mathsf{D}^{b}\!\bigl(\operatorname{Coh}(\mathbb{P}^{d})\bigr)\),
where \(\mathcal{O}=\mathcal{O}_{\mathbb{P}^{d}}\) is the structure sheaf
and the objects \(\mathcal{O}\langle-i\rangle\) are the corresponding
Serre twisting line bundles on \(\mathbb{P}^{d}\).
Applying Bondal's theorem~\cite{12} yields the derived equivalence
\[
\mathsf{D}^{b}\!\bigl(\operatorname{Coh}(\mathbb{P}^{d})\bigr)
\simeq
\mathsf{D}^{b}\!\bigl(B_d\textup{-}\mathrm{mod}\bigr),
\]
where \(B_d\) is the endomorphism algebra of the corresponding exceptional
sequence. We now show how this equivalence follows from Theorem~4.6.

\begin{example}
Consider the polynomial algebra
\[
S=
k\langle x_0,x_1,\ldots,x_d\rangle
\big/
\langle x_i x_j-x_j x_i \mid 0\leq i<j\leq d\rangle.
\]
This is an AS-regular algebra of global dimension \(d+1\), and hence is
generalized AS-regular. Its Koszul dual is the exterior algebra
\[
E=\bigwedge V,
\qquad
V=k^{d+1},
\]
which is finite-dimensional and self-injective. Moreover,
\[
\operatorname{Coh}(\mathbb{P}^{d})
\cong
\operatorname{qgr}(S).
\]

Hence, by Theorem~4.6,
\[
\mathsf{D}^{b}\!\bigl(\operatorname{qgr}(S)\bigr)
\simeq
\mathsf{D}^{b}\!\bigl(B\textup{-}\mathrm{mod}\bigr).
\]

As a triangulated category,
\(\mathsf{D}^{b}\!\bigl(\operatorname{qgr}(S)\bigr)\) is generated by the
full strong exceptional sequence
\[
\pi S,\;\pi S\langle-1\rangle,\;\pi S\langle-2\rangle,\ldots,\pi S\langle-d\rangle.
\]

The quiver \(Q^{\mathbb Z}\) of the \(\mathbb Z\)-category
\(S^{\mathbb Z}\) has vertices \((x,m)\), \(m\in\mathbb Z\), and between
two consecutive vertices \((x,m)\) and \((x,m+1)\) there are \(d+1\)
arrows
\[
\begin{tikzcd}[column sep=1.35cm]
\cdots
\arrow[r,bend left=18,"{(x_0,-2)}"]
\arrow[r,"\vdots" description]
\arrow[r,bend right=18,"{(x_d,-2)}"']
&
(x,-1)
\arrow[r,bend left=18,"{(x_0,-1)}"]
\arrow[r,"\vdots" description]
\arrow[r,bend right=18,"{(x_d,-1)}"']
&
(x,0)
\arrow[r,bend left=18,"{(x_0,0)}"]
\arrow[r,"\vdots" description]
\arrow[r,bend right=18,"{(x_d,0)}"']
&
(x,1)
\arrow[r,bend left=18,"{(x_0,1)}"]
\arrow[r,"\vdots" description]
\arrow[r,bend right=18,"{(x_d,1)}"']
&
(x,2)
\arrow[r,draw=none,"\cdots" description]
&
{}
\end{tikzcd}
\]

The relations are induced by the commutativity relations of \(S\), namely
\[
(x_j,m+1)(x_i,m)
-
(x_i,m+1)(x_j,m)
\]
for all \(m\in\mathbb Z\) and \(0\leq i<j\leq d\).

The full subcategory of \(S^{\mathbb Z}\) on the vertices
\[
(x,-d),(x,-d+1),\ldots,(x,-1),(x,0)
\]
is convex. The corresponding convex truncation algebra \(B\) has quiver
\[
\begin{tikzcd}[column sep=2.5cm]
(x,-d)
\arrow[r,bend left=18,"{(x_0,-d)}"]
\arrow[r,"\vdots" description]
\arrow[r,bend right=18,"{(x_d,-d)}"']
&
(x,-d+1)
\arrow[r,bend left=18,"{(x_0,-d+1)}"]
\arrow[r,"\vdots" description]
\arrow[r,bend right=18,"{(x_d,-d+1)}"']
&
\cdots
\arrow[r,bend left=18,"{(x_0,-1)}"]
\arrow[r,"\vdots" description]
\arrow[r,bend right=18,"{(x_d,-1)}"']
&
(x,0),
\end{tikzcd}
\]
subject to the relations
\[
(x_j,m+1)(x_i,m)
-
(x_i,m+1)(x_j,m),
\qquad
-d\leq m\leq -2,\quad 0\leq i<j\leq d.
\]

Thus \(B\) is the convex truncation of \(S^{\mathbb Z}\) determined by the
\(d+1\) consecutive vertices
\[
(x,-d),\ldots,(x,0).
\]

Since \(B\) is Koszul, its Koszul dual \(B^{!}\) is the quadratic dual of
\(B\). Let
\[
y_0,y_1,\ldots,y_d
\]
denote the basis dual to
\[
x_0,x_1,\ldots,x_d.
\]
The quiver of \(B^{!}\) has the same vertices as \(B\), with \(d+1\)
arrows between consecutive vertices, now oriented in the opposite
direction:
\[
\begin{tikzcd}[column sep=2.5cm]
(x,-d)
&
(x,-d+1)
\arrow[l,bend right=18,"{(y_0,-d)}"']
\arrow[l,"\vdots" description]
\arrow[l,bend left=18,"{(y_d,-d)}"]
&
\cdots
\arrow[l,bend right=18,"{(y_0,-d+1)}"']
\arrow[l,"\vdots" description]
\arrow[l,bend left=18,"{(y_d,-d+1)}"]
&
(x,0)
\arrow[l,bend right=18,"{(y_0,-1)}"']
\arrow[l,"\vdots" description]
\arrow[l,bend left=18,"{(y_d,-1)}"].
\end{tikzcd}
\]

The quadratic relations of \(B^{!}\) are induced by those of the exterior
algebra \(E=\bigwedge V\). Thus, for every admissible \(m\), they are
\[
(y_j,m)(y_i,m+1)
+
(y_i,m)(y_j,m+1),
\qquad
0\leq i<j\leq d,
\]
together with
\[
(y_i,m)(y_i,m+1),
\qquad
0\leq i\leq d.
\]

Equivalently, \(B^{!}\) is the finite convex truncation of the
\(\mathbb Z\)-category associated with the exterior algebra
\[
E=\bigwedge V
\]
on the vertices
\[
(x,-d),\ldots,(x,0).
\]

Consequently, applying Theorem~4.6 to the Koszul dual truncation gives
\[
\mathsf{D}^{b}\!\bigl(\operatorname{qgr}(S)\bigr)
\cong
\mathsf{D}^{b}\!\bigl(B^{!}\textup{-}\mathrm{mod}\bigr).
\]
Thus the bounded derived category of coherent sheaves on projective
space admits the description
\[
\mathsf{D}^{b}\!\bigl(\operatorname{Coh}(\mathbb P^{d})\bigr)
\cong
\mathsf{D}^{b}\!\bigl(B^{!}\textup{-}\mathrm{mod}\bigr).
\]
\end{example}
We conclude the paper with a further consequence due to Minamoto~\cite{45},
which has been regarded as a noncommutative Beilinson-type theorem.
\begin{example} We consider, for \(n\geq 2\), the algebra \[ \Gamma= k\langle x_1,x_2,\ldots,x_n\rangle \big/ \langle \textstyle\sum_{i=1}^{n}x_i^2\rangle, \] which is generalized AS-regular and Koszul. Its Koszul dual is the algebra \[ \Gamma^{!}= k\langle x_1,x_2,\ldots,x_n\rangle \Big/ \left\langle x_i x_j\ (i\neq j),\; x_i^2-x_2^2\ (i\neq 2) \right\rangle, \] which is finite-dimensional and self-injective~\cite{42,58}. Since \(\operatorname{gldim}\Gamma=2\), Theorem~4.6 implies that \(\mathsf{D}^{b}\!\bigl(\operatorname{qgr}(\Gamma)\bigr)\) admits the full strong exceptional sequence \[ \pi\Gamma,\;\pi\Gamma\langle-1\rangle. \] The quiver of the \(\mathbb Z\)-category \(\Gamma^{\mathbb Z}\) has one vertex \((x,m)\) at each \(m\in\mathbb Z\), with \(n\) arrows \[ (x_i,m)\colon (x,m)\longrightarrow (x,m+1), \qquad 1\leq i\leq n. \] Thus a portion of its quiver is \[ \begin{tikzcd}[column sep=2.2cm] \cdots \arrow[r,bend left=18,"{(x_1,-2)}"] \arrow[r,"\vdots" description] \arrow[r,bend right=18,"{(x_n,-2)}"'] & (x,-1) \arrow[r,bend left=18,"{(x_1,-1)}"] \arrow[r,"\vdots" description] \arrow[r,bend right=18,"{(x_n,-1)}"'] & (x,0) \arrow[r,bend left=18,"{(x_1,0)}"] \arrow[r,"\vdots" description] \arrow[r,bend right=18,"{(x_n,0)}"'] & (x,1) \arrow[r,draw=none,"\cdots" description] & {} \end{tikzcd} \] subject to the relations \[ \sum_{i=1}^{n}(x_i,m+1)(x_i,m), \qquad m\in\mathbb Z. \] The algebra \(B\) occurring in Theorem~4.6 is the convex truncation of \(\Gamma^{\mathbb Z}\) determined by the two consecutive vertices \((x,-1)\) and \((x,0)\). Hence its quiver is \[ \begin{tikzcd}[column sep=3cm] (x,-1) \arrow[r,bend left=22,"{(x_1,-1)}"] \arrow[r,bend left=10,"{(x_2,-1)}"] \arrow[r,"\vdots" description] \arrow[r,bend right=22,"{(x_n,-1)}"'] & (x,0). \end{tikzcd} \] Since this truncation contains no paths of length two, \(B\) has no nontrivial relations. In particular, \(B\) is the path algebra of the \(n\)-Kronecker quiver. Therefore, Theorem~4.6 yields \[ \mathsf{D}^{b}\!\bigl(\operatorname{qgr}(\Gamma)\bigr) \simeq \mathsf{D}^{b}\!\bigl(B\textup{-}\mathrm{mod}\bigr). \] Since \(B\) is Koszul, we may also consider its Koszul dual \(B^{!}\). The quiver of \(B^{!}\) has the same two vertices, with the \(n\) arrows oriented in the opposite direction: 
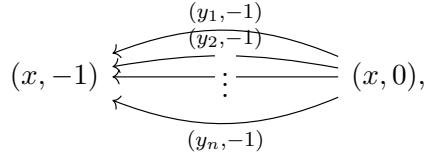
\[ \begin{tikzcd}[column sep=3cm] (x,-1) & (x,0) \arrow[l,bend right=22,"{(y_1,-1)}"'] \arrow[l,bend right=10,"{(y_2,-1)}"'] \arrow[l,"\vdots" description] \arrow[l,bend left=22,"{(y_n,-1)}"], \end{tikzcd} \] where \(y_1,\ldots,y_n\) denote the arrows dual to \(x_1,\ldots,x_n\). There are again no paths of length two in this quiver. Consequently, \(B^{!}\) has no nontrivial quadratic relations and is the path algebra of the oppositely oriented \(n\)-Kronecker quiver. Thus \[ B^{!}\cong B^{\mathrm{op}}, \] and, after interchanging the two vertices, \(B^{!}\) is isomorphic to \(B\). Hence Theorem~4.6 also gives \[ \mathsf{D}^{b}\!\bigl(\operatorname{qgr}(\Gamma)\bigr) \simeq \mathsf{D}^{b}\!\bigl(B^{!}\textup{-}\mathrm{mod}\bigr). \] Thus, in this example, both \(B\) and its Koszul dual \(B^{!}\) are \(n\)-Kronecker algebras, with opposite orientations, and both provide a finite-dimensional algebraic model for \(\mathsf{D}^{b}\!\bigl(\operatorname{qgr}(\Gamma)\bigr)\). 
\end{example}

% --------------------------------------------------

\addcontentsline{toc}{section}{References}

\begin{thebibliography}{99} \bibitem{1} M.~Artin and W.~Schelter, \emph{Graded algebras of global dimension 3}, Adv. Math. \textbf{66} (1987), 171--216. \bibitem{2} M.~Artin and J.~J.~Zhang, \emph{Noncommutative projective schemes}, Adv. Math. \textbf{109} (1994), 228--287. \bibitem{3} L.~L.~Avramov and D.~Eisenbud, \emph{Regularity of modules over a Koszul algebra}, J. Algebra \textbf{153} (1992), 85--90. \bibitem{4} J.~Backelin and R.~Fr\"oberg, \emph{Koszul algebras, Veronese subrings and rings with linear resolutions}, Rev. Roumaine Math. Pures Appl. \textbf{30} (1985), 85--97. \bibitem{5} J.~Backelin, \emph{On the rates of growth of the homologies of Veronese subrings}, in \emph{Algebra, Algebraic Topology, and Their Interactions}, J.-E.~Roos (ed.), Lecture Notes in Mathematics, vol.~1183, Springer, Berlin, 1986, pp.~79--100. \bibitem{6} A.~A.~Beilinson, \emph{Coherent sheaves on \(\mathbb{P}^{n}\) and problems of linear algebra}, Funct. Anal. Appl. \textbf{12} (1978), 214--216. \bibitem{7} A.~A.~Beilinson, V.~Ginzburg, and V.~V.~Schechtman, \emph{Koszul duality}, J. Geom. Phys. \textbf{5} (1988), 317--350. \bibitem{8} A.~Beilinson, V.~Ginzburg, and W.~Soergel, \emph{Koszul duality patterns in representation theory}, J. Amer. Math. Soc. \textbf{9} (1996), no.~2, 473--527. \bibitem{9} A.~Beligiannis, \emph{The homological theory of contravariantly finite subcategories: Auslander--Buchweitz contexts, Gorenstein categories and (co-)stabilization}, Comm. Algebra \textbf{28} (2000), no.~10, 4547--4596. \bibitem{10} I.~N.~Bernstein, I.~M.~Gel'fand, and S.~I.~Gel'fand, \emph{Algebraic bundles over \(\mathbb{P}^{n}\) and problems of linear algebra}, Funct. Anal. Appl. \textbf{12} (1978), 212--214. \bibitem{11} R.~B{\o}gvad and S.~Halperin, \emph{On a conjecture of Roos}, in \emph{Algebra, Algebraic Topology and Their Interactions (Stockholm, 1983)}, Lecture Notes in Mathematics, vol.~1183, Springer, Berlin, 1986, pp.~120--127. \bibitem{12} A.~I.~Bondal, \emph{Representation of associative algebras and coherent sheaves}, Math. USSR-Izv. \textbf{34} (1990), no.~1, 23--42. \bibitem{13} A.~I.~Bondal and M.~M.~Kapranov, \emph{Representable functors, Serre functors, and mutations}, Math. USSR-Izv. \textbf{35} (1990), no.~3, 519--541. \bibitem{14} A.~Bondal and M.~Van den Bergh, \emph{Generators and representability of functors in commutative and noncommutative geometry}, Mosc. Math. J. \textbf{3} (2003), no.~1, 1--36. \bibitem{15} A.~M.~Bouhada, \emph{A non-graded Koszul duality and its applications}, preprint (2026), available at \url{https://arxiv.org/abs/2604.16805}. \bibitem{16}
A.~M.~Bouhada,
\emph{Koszul duality for finite-dimensional absolutely Koszul algebras},
preprint (2026), arXiv:2609.06768, available at
\url{https://arxiv.org/abs/2609.06768}. \bibitem{17} A.~M.~Bouhada, \emph{Koszul duality for quadratic monomial algebras}, preprint (2026), arXiv:2604.20177, available at \url{https://arxiv.org/abs/2604.20177}. \bibitem{18} A.~M.~Bouhada, M.~Huang, and S.~Liu, \emph{Koszul duality for non-graded derived categories}, preprint (2019), available at \url{https://arxiv.org/abs/1908.06153}. \bibitem{19} A.~M.~Bouhada, M.~Huang, Z.~Lin, and S.~Liu, \emph{A representation theoretic perspective of Koszul theory}, preprint (2024), available at \url{https://arxiv.org/abs/2411.15449}. \bibitem{20} R.~O.~Buchweitz, \emph{Maximal Cohen--Macaulay Modules and Tate Cohomology over Gorenstein Rings}, with appendices by L.~L.~Avramov, B.~Briggs, S.~B.~Iyengar, and J.~C.~Letz, Mathematical Surveys and Monographs, vol.~262, American Mathematical Society, Providence, RI, 2021. \bibitem{21} X.-W.~Chen, \emph{Relative singularity categories and Gorenstein-projective modules}, Math. Nachr. \textbf{284} (2011), no.~2--3, 199--212. \bibitem{22} A.~Conca, E.~De Negri, and M.~E.~Rossi, \emph{Koszul algebras and regularity}, in \emph{Commutative Algebra: Expository Papers Dedicated to David Eisenbud on the Occasion of His 65th Birthday}, I.~Peeva (ed.), Springer, New York, 2013, pp.~285--315. \bibitem{23} A.~Conca, S.~B.~Iyengar, H.~D.~Nguyen, and T.~R\"omer, \emph{Absolutely Koszul algebras and the Backelin--Roos property}, Acta Math. Vietnam. \textbf{40} (2015), no.~3, 353--374. \bibitem{24} D.~Eisenbud, G.~Fl{\o}ystad, and F.-O.~Schreyer, \emph{Sheaf cohomology and free resolutions over exterior algebras}, Trans. Amer. Math. Soc. \textbf{355} (2003), 4397--4426. \bibitem{25} D.~Eisenbud and S.~Goto, \emph{Linear free resolutions and minimal multiplicity}, J. Algebra \textbf{88} (1984), 89--133. \bibitem{26} D.~Eisenbud and C.~Huneke (eds.), \emph{Free Resolutions in Commutative Algebra and Algebraic Geometry}, A K Peters, Wellesley, MA, 1992. \bibitem{27} D.~Eisenbud and F.-O.~Schreyer, \emph{Resultants and Chow forms via exterior syzygies}, J. Amer. Math. Soc. \textbf{16} (2003), no.~3, 537--579. \bibitem{28} D.~Eisenbud and F.-O.~Schreyer, \emph{Betti numbers of graded modules and cohomology of vector bundles}, J. Amer. Math. Soc. \textbf{22} (2009), no.~3, 859--888. \bibitem{29} W.~L.~Gan and L.~Li, \emph{Koszulity of directed categories in representation stability theory}, J. Algebra \textbf{501} (2018), 88--129. \bibitem{30} S.~I.~Gel'fand and Yu.~I.~Manin, \emph{Methods of Homological Algebra}, Springer-Verlag, Berlin, 1996. \bibitem{31} J.~Herzog and S.~B.~Iyengar, \emph{Koszul modules}, J. Pure Appl. Algebra \textbf{201} (2005), no.~1--3, 154--188. \bibitem{32} S.~B.~Iyengar and T.~R\"omer, \emph{Linearity defects of modules over commutative rings}, J. Algebra \textbf{322} (2009), no.~9, 3212--3237. \bibitem{33} P.~J{\o}rgensen, \emph{Linear free resolutions over non-commutative algebras}, Compos. Math. \textbf{140} (2004), no.~4, 1053--1058. \bibitem{34} P.~J{\o}rgensen, \emph{A noncommutative BGG correspondence}, Pacific J. Math. \textbf{218} (2005), no.~2, 357--377. \bibitem{35} M.~M.~Kapranov, \emph{On the derived category of coherent sheaves on Grassmann manifolds}, Math. USSR-Izv. \textbf{24} (1985), no.~1, 183--192. \bibitem{36} M.~M.~Kapranov, \emph{On the derived categories of coherent sheaves on some homogeneous spaces}, Invent. Math. \textbf{92} (1988), no.~3, 479--508. \bibitem{37} M.~Kashiwara and P.~Schapira, \emph{Categories and Sheaves}, Grundlehren der Mathematischen Wissenschaften, vol.~332, Springer-Verlag, Berlin--Heidelberg, 2006. \bibitem{38} B.~Keller and D.~Vossieck, \emph{Sous les cat\'egories d\'eriv\'ees}, C. R. Acad. Sci. Paris S\'er. I Math. \textbf{305} (1987), no.~6, 225--228. \bibitem{39} H.~Li and Q.~Wu, \emph{Generalized Koszul algebra and Koszul duality}, J. Algebra \textbf{619} (2023), 643--685. \bibitem{40} R.~Mart\'inez-Villa, \emph{Graded, selfinjective, and Koszul algebras}, J. Algebra \textbf{215} (1999), no.~1, 34--72. \bibitem{41} R.~Mart\'inez-Villa and M.~Saor\'in, \emph{Koszul equivalences and dualities}, Pacific J. Math. \textbf{214} (2004), no.~2, 359--378. \bibitem{42} R.~Mart\'inez-Villa and O.~Solberg, \emph{Artin--Schelter regular algebras and categories}, J. Pure Appl. Algebra \textbf{215} (2011), no.~4, 546--565. \bibitem{43} R.~Mart\'inez-Villa and D.~Zacharia, \emph{Approximations with modules having linear resolutions}, J. Algebra \textbf{266} (2003), no.~2, 671--697. \bibitem{44} V.~Mazorchuk, S.~Ovsienko, and C.~Stroppel, \emph{Quadratic duals, Koszul dual functors, and applications}, Trans. Amer. Math. Soc. \textbf{361} (2009), no.~3, 1129--1172. \bibitem{45} H.~Minamoto, \emph{A noncommutative version of Beilinson's theorem}, J. Algebra \textbf{320} (2008), 238--252. \bibitem{46} J.~Miyachi, \emph{Localization of triangulated categories and derived categories}, J. Algebra \textbf{141} (1991), no.~2, 463--483. \bibitem{47} D.~Orlov, \emph{Triangulated categories of singularities and D-branes in Landau--Ginzburg models}, Trudy Mat. Inst. Steklov. \textbf{246} (2004), 240--262. \bibitem{48} D.~Orlov, \emph{Triangulated categories of singularities, and equivalences between Landau--Ginzburg models}, Mat. Sb. \textbf{197} (2006), no.~12, 117--132. \bibitem{49} D.~Orlov, \emph{Derived categories of coherent sheaves and triangulated categories of singularities}, in \emph{Algebra, Arithmetic, and Geometry: In Honor of Yu.~I.~Manin}, Vol.~II, Progr. Math., vol.~270, Birkh\"auser Boston, Boston, MA, 2009, pp.~503--531. \bibitem{50} I.~Reiten and M.~Van den Bergh, \emph{Noetherian hereditary abelian categories satisfying Serre duality}, J. Amer. Math. Soc. \textbf{15} (2002), no.~2, 295--366. \bibitem{51} D.~Rogalski, \emph{An introduction to noncommutative projective geometry}, preprint (2014), available at \url{https://arxiv.org/abs/1403.3065}. \bibitem{52} J.-P.~Serre, \emph{Faisceaux alg\'ebriques coh\'erents}, Ann. of Math. (2) \textbf{61} (1955), no.~2, 197--278. \bibitem{53} J.-P.~Serre, \emph{Fiber spaces}, in Russian, IL, Moscow, 1957, pp.~372--453. \bibitem{54} I.~Shipman, \emph{A geometric approach to Orlov's theorem}, Compos. Math. \textbf{148} (2012), no.~5, 1365--1389. \bibitem{55} C.~A.~Weibel, \emph{An Introduction to Homological Algebra}, Cambridge Studies in Advanced Mathematics, vol.~38, Cambridge University Press, Cambridge, 1994. \bibitem{56} K.~Yanagawa, \emph{Castelnuovo--Mumford regularity for complexes and weakly Koszul modules}, J. Pure Appl. Algebra \textbf{213} (2009), no.~3, 399--413. \bibitem{57} K.~Yanagawa, \emph{Linearity defect and regularity over a Koszul algebra}, Math. Scand. \textbf{104} (2009), no.~2, 273--290. \bibitem{58} J.~J.~Zhang, \emph{Quadratic algebras with few relations}, Glasgow Math. J. \textbf{39} (1997), no.~3, 323--332. \end{thebibliography}
\end{document}